\documentclass[journal]{IEEEtran}
\usepackage{amsmath}    
\usepackage{graphics,epsfig,subfigure}    
\usepackage{verbatim}   
\usepackage{color}      
\usepackage{subfigure}  
\usepackage{amssymb}
\usepackage{wasysym}
\usepackage{epstopdf}
\usepackage{graphicx}
\usepackage{mathrsfs}

\newcommand{\prob}[1]{\text{Pr}\big\{#1\big\}}
\newcommand{\expect}[1]{\mathbb{E}\big\{#1\big\}}

\newcommand{\bv}[1]{{\boldsymbol{#1} }}
\newcommand{\script}[1]{{{\cal{#1} }}}

\newcommand{\rmax}{\text{argmax}}

\newtheorem{lemma}{\textbf{Lemma}}

\newtheorem{theorem}{\textbf{Theorem}}

\newtheorem{assumption}{\textbf{Assumption}}
\newtheorem{definition}{\textbf{Definition}}

\allowdisplaybreaks

\allowdisplaybreaks
\begin{document}

\title{Fast-Convergent Learning-aided Control in Energy Harvesting Networks}
\author{ Longbo Huang\\
longbohuang@tsinghua.edu\\
IIIS, Tsinghua University  
\thanks{Longbo Huang  (http://www.iiis.tsinghua.edu.cn/$\sim$huang) is with the Institute for Theoretical Computer Science and the Institute for Interdisciplinary Information Sciences, Tsinghua University, Beijing, P. R. China. }
}


\maketitle

\begin{abstract}
In this paper, we present a novel learning-aided energy management scheme  ($\mathtt{LEM}$)  for multihop energy harvesting networks. Different from prior works on this problem, our algorithm explicitly incorporates  information learning  into system control via a step called \emph{perturbed dual learning}. $\mathtt{LEM}$ does not require any statistical information of the system dynamics for implementation, and efficiently resolves the challenging energy outage problem. We show that $\mathtt{LEM}$ achieves the near-optimal $[O(\epsilon), O(\log(1/\epsilon)^2)]$ utility-delay tradeoff with an $O(1/\epsilon^{1-c/2})$ energy buffers ($c\in(0,1)$). More interestingly,  $\mathtt{LEM}$ possesses a \emph{convergence time} of $O(1/\epsilon^{1-c/2} +1/\epsilon^c)$, which is much faster than the $\Theta(1/\epsilon)$ time of pure queue-based techniques or the $\Theta(1/\epsilon^2)$ time of approaches that rely purely on learning the system statistics. 
This fast convergence property makes $\mathtt{LEM}$ more adaptive and efficient in resource allocation in dynamic environments. 
The design and analysis of $\mathtt{LEM}$ demonstrate how system control algorithms can be augmented by  learning and what the benefits are. The methodology and algorithm can also be applied to similar problems, e.g., processing networks, where nodes require nonzero amount of contents to support their actions. 
\end{abstract} 

\section{Introduction}
Recent developments in energy harvesting technologies make it possible for wireless devices to support their functions by harvesting energy from the environment.  For example, by using solar panels \cite{solar-energy-05} \cite{solar-source-2014}, by harvesting ambient radio power \cite{enhant-mobicom09}, and by converting mechanical  vibration into energy \cite{vibrate-to-energy-01}, \cite{vibrational-source-2014}. Due to the capability in providing long lasting energy supply, the energy harvesting technology has the potential to become a promising solution to energy problems in networks formed by self-powered devices, e.g., wireless sensor networks and mobile devices.

To realize the full benefits of   energy harvesting,  algorithms must be designed to efficiently incorporate it into system control.  
In this paper, we develop an \emph{online learning-aided energy management} scheme for energy harvesting networks. Specifically, we consider a discrete stochastic network, where network links have time-varying qualities, and nodes are powered by finite capacity energy storage devices and can harvest energy from the environment. 
In each time slot, every node decides how much new workload to admit, e.g., sampled data from a field, and how much power to spend for traffic transmission (or data processing). The objective of the network is to find a joint energy management and scheduling policy, so as to maximize the aggregate traffic utility, while ensuring network stability and energy availability, i.e., the network nodes always have enough energy to support  transmission. 

There have been many previous works  on energy harvesting networks. Works \cite{power-harvesting-kansal07} and \cite{adaptive-cycling-07} consider a leaky-bucket like structure and design joint energy prediction and power management schemes for energy harvesting sensor nodes. \cite{energy-tradeoff-10} focuses on designing energy-efficient schemes that maximize the decay exponent of the queue size. \cite{control-rechargeable-twc10} develops scheduling algorithms to achieve near-optimal utility for energy harvesting networks with time-varying channels.  
\cite{power-routing-lin-infocom05} designs an energy-aware routing scheme  that  achieves optimality as the network size increases. 
\cite{huangneely-energy-ton13} proposes an online energy management and scheduling algorithm for multihop energy harvesting networks. \cite{tap-energy-ton14} considers joint compression and transmission in energy harvesting networks.  \cite{chen-energy-ton14} considers a multihop network and proposes a control scheme based on energy replenishment rate estimation. 

However, we notice that the aforementioned works either focus on scenarios where complete statistical information is given beforehand, or try to design schemes that do not require such information. Therefore, they ignore the potential benefits of utilizing information of system dynamics in control, and do not provide  interfaces for integrating  information collecting and learning techniques \cite{pattern-recognition-bishop}, e.g., sensing and data mining or machine learning, into algorithm design. 
%
In this work, we try to explicitly bring information learning into the system control framework. Specifically, we develop a learning mechanism called  \emph{perturbed dual learning} and propose a learning-aided energy management scheme ($\mathtt{LEM}$).  

$\mathtt{LEM}$ is an online control algorithm and \emph{does not require any statistical information} for implementation. Instead, it  builds an empirical distribution of the system dynamics, including network condition variation and energy availability fluctuation. Then, it learns an approximate optimal Lagrange multiplier of a carefully constructed underlying optimization problem that captures system optimality,  via a step called \emph{perturbed dual learning}. Finally, $\mathtt{LEM}$ incorporates the learned information into the system controller by augmenting the controller with the approximate multiplier. 
%
We show that $\mathtt{LEM}$ is able to achieve a near-optimal $[O(\epsilon), O(\log(\frac{1}{\epsilon})^2)]$ utility-delay tradeoff for general multihop energy harvesting networks with an $O((\frac{1}{\epsilon})^{2/3}\log(\frac{1}{\epsilon})^2)$ energy storage capacity and resolves the energy outage problem. 
Moreover, we show that by incorporating information learning, one can significantly improve the algorithm \emph{convergence time}, i.e., the time an algorithm takes to converge to its optimal operating point: $\mathtt{LEM}$ requires an $O((\frac{1}{\epsilon})^{2/3}\log(\frac{1}{\epsilon})^2)$ time for convergence, whereas existing queue-based algorithms require a $\Theta(1/\epsilon)$ time and algorithms based purely on learning the statistics require a $\Theta(1/\epsilon^2)$ time.   
This fast convergence implies that learning-aided algorithms can adapt  faster when the environment statistics changes, which indicates better robustness and higher efficiency in resource allocation. 

Learning-aided control with dual learning was first developed in \cite{huang-learning-sig-14}. In this work, we extend the results to resolve  energy outage problems  in energy harvesting networks via a perturbed version of dual learning. Intuitively speaking, perturbed dual learning learns a perturbed empirical optimal Lagrange multiplier required for ``no-underflow'' systems, where optimal multipliers must be steered  and made trackable by queues. 

Our paper is mostly related to recent works \cite{power-harvesting-kansal07}, \cite{chen-energy-ton14}, and \cite{huangneely-energy-ton13}. Specifically, both \cite{power-harvesting-kansal07} and \cite{chen-energy-ton14} try to form estimations of the harvestable energy rates and utilize the information in network control. However, they do not consider the system dynamics and do not explicitly characterize network delay performance. On the other hand, \cite{huangneely-energy-ton13} focuses on achieving long term performance guarantees without learning. Moreover, these three works do not characterize the algorithm convergence speed, which is an important metric for measuring  the efficiency of control algorithms in learning the optimal system operating point in dynamic environments. 

We summarize the main contributions as follows:  
\begin{itemize} 
\item We propose the Learning-aided Energy Management algorithm ($\mathtt{LEM}$) for multihop energy harvesting networks, and show that  $\mathtt{LEM}$ achieves a near-optimal $[O(\epsilon), O(\log(\frac{1}{\epsilon})^2)]$ utility-delay tradeoff with an $O((\frac{1}{\epsilon})^{2/3}\log(\frac{1}{\epsilon})^2)$ energy storage capacity.  
\item We show that $\mathtt{LEM}$ possesses an $O((\frac{1}{\epsilon})^{2/3}\log(\frac{1}{\epsilon})^2)$ convergence time. This convergence time is 
much faster compared to the $\Theta(\frac{1}{\epsilon})$ time of existing queue-based techniques and the $\Theta(\frac{1}{\epsilon})^2$ time required for approaches that purely rely on learning the statistics.
\item We analyze the performance of $\mathtt{LEM}$ with the \emph{augmented drift analysis} approach, which handles the interplay between learning and control and no-underflow constraints.  This analysis approach can likely find applications to other similar problems with the no-underflow constraints, e.g., processing networks \cite{jiang-spn}.  


\end{itemize}



The rest of the paper is organized as follows.  We present the system model in  Section \ref{section:model}. We  explain the algorithm design approach and present the $\mathtt{LEM}$ algorithm in Section \ref{section:alg-dual-learning} and explain the intuition. Then, we present the performance results of $\mathtt{LEM}$ in Section \ref{section:analysis}. Simulation results are provided in Section \ref{section:sim}. We conclude the paper in Section \ref{section:con}.

\vspace{-.08in}
\section{The System Model}\label{section:model}
We consider a general  multi-hop network that operates in slotted time. The network is modeled by a directed graph $\script{G}=(\script{N}, \script{L})$, where $\script{N}=\{1, 2, ..., N\}$ is the set of  nodes in the network, and $\script{L}=\{[n, m], \,\, n, m\in\script{N}\}$ is the set of communication links. We use $\script{N}_n^{(o)}$ to denote the set of nodes $b$ with $[n, b]\in\script{L}$ for each node $n$,  and use $\script{N}_n^{(in)}$ to denote the set of nodes $a$ with $[a, n]\in\script{L}$. We define $d_{\max}\triangleq\max_n(|\script{N}^{(in)}_n|, |\script{N}^{(o)}_n|)$ the maximum in-degree/out-degree that any node $n$ can have. 



\subsection{The Traffic and Utility Model}\label{subsection:rate-utility}

At every time slot, the network decides how much new workload (called packets below) destined for node $c$ to admit at node $n$. We call this traffic the \emph{commodity $c$} data and use $R^{(c)}_n(t)$ to denote the amount of new commodity $c$ data admitted. We assume that $0\leq R^{(c)}_n(t)\leq R_{\max}$ for all $n, c$ with some finite $R_{\max}>0$ at all time. 

We assume that each commodity is associated with a utility function $U^{(c)}_{n}(\overline{r}^{nc})$, where $\overline{r}^{nc}$ is the time average rate of the commodity $c$ traffic admitted into node $n$, defined as $\overline{r}^{nc}=\lim_{t\rightarrow\infty}\frac{1}{t}\sum_{\tau=0}^{t-1}\expect{R^{(c)}_n(\tau)}$.  \footnote{In this paper, we assume for clarity that all  limits exist with probability $1$.  When some limits do not exist, we can obtain similar results  by replacing limit by $\liminf$ or $\limsup$, but the results are more involved. } 
Each $U^{(c)}_{n}(r)$ function is assumed to be increasing, continuously differentiable, and concave in $r$ with a bounded first derivative and $U^{(c)}_{n}(0)=0$.  We define $\beta\triangleq\max_{n, c}(U^{(c)}_{n})'(0)$ the maximum first derivative of all utility functions. 

\subsection{The Transmission Model}
In order to deliver the admitted data to their destinations, each node needs to allocate power to the links for transmission at every time slot. To model the effect that the transmission rates typically also depend on the link conditions and that the link conditions may be time-varying, we denote $\bv{S}(t)$ the network \emph{channel state}, i.e.,  the $N$-by-$N$ matrix where the $(n, m)$ component of $\bv{S}(t)$ denotes the channel condition between nodes $n$ and $m$. 

Denote  $P_{[n, b]}(t)$ the power allocated to link $[n, b]$ at time $t$. At every time slot, if $\bv{S}(t)=s_i$, the power allocation vector $\bv{P}(t)=(P_{[n, b]}(t), [n, b] \in\script{L})$ must be chosen from some feasible power allocation set $\script{P}^{(s_i)}$. We assume that $\script{P}^{(s_i)}$ is compact for all $s_i$, and that every power vector in $\script{P}^{(s_i)}$ satisfies the constraint that for each node $n$, $0\leq \sum_{b\in\script{N}^{(o)}_n}P_{[n, b]}(t)\leq P_{\max}$ for some finite $P_{\max}>0$. We also assume that for any $\bv{P}\in\script{P}^{(s_i)}$, setting the entry $P_{[n, b]}$ to zero yields another power vector that is still in $\script{P}^{(s_i)}$. 
Given channel state $\bv{S}(t)$ and  power allocation vector $\bv{P}(t)$, the transmission rate over link $[n, b]$ is given by the rate-power function $\mu_{[n, b]}(t)=\mu_{[n, b]}(\bv{S}(t),\bv{P}(t))$. 

For each $s_i$, we  assume that the function $\mu_{[n, b]}(s_i, \bv{P})$ satisfies the following properties:  Let $\bv{P}, \bv{P}'\in\script{P}^{(s_i)}$ be such that  $\bv{P}'$ is obtained by changing any single component $P_{[n, b]}$ in $\bv{P}$ to zero. Then, (i) there exists some finite constant $\kappa>0$ that: 
\begin{eqnarray}
\mu_{[n, b]}(s_i, \bv{P})\leq \mu_{[n, b]}(s_i, \bv{P}') + \kappa P_{[n, b]}, \label{eq:rate-property1}
\end{eqnarray}
and (ii) for each link $[a, m]\neq[n, b]$, 
\begin{eqnarray}
\hspace{-.1in}\mu_{[a, m]}(s_i, \bv{P})\leq \mu_{[a, m]}(s_i, \bv{P}'). \label{eq:rate-property2}
\end{eqnarray}
These properties  can be satisfied by most rate-power functions, e.g.,  when the rate function is differentiable and has finite directional derivatives with respect to power \cite{neelyenergy}, and when link rates do not improve with increased interference. 

We assume that there exists a finite constant $\mu_{\max}$ such that $\mu_{[n, b]}(t)\leq\mu_{\max}$ for all time under any power allocation vector $\bv{P}(t)$ and any  channel state $\bv{S}(t)$. 
We use $\mu_{[n, b]}^{(c)}(t)$ to denote the rate allocated to the commodity $c$ data over link $[n, b]$ at time $t$. It can be seen that $\sum_{c}\mu_{[n, b]}^{(c)}(t)\leq\mu_{[n, b]}(t)$ for all  $[n, b]$ and for all $t$. 


\subsection{The  Energy Harvesting Model}
Each node in the network is assumed to be powered by a \emph{finite} capacity energy storage device, e.g., a battery or an ultra-capacitor \cite{opt-energy-twc10}. We model such a device with an  \emph{energy queue}. We use the energy queue size at node $n$ at time $t$, denoted by $E_n(t)$, to measure the amount of the energy stored at node $n$ at time $t$. 
Each node $n$ can observe its  current energy level $E_n(t)$. 
In any time slot $t$, the power allocation vector $\bv{P}(t)$ must satisfy the following ``energy-availability'' constraint: 
\footnote{We measure time in unit size slots, so that our power $P_{[n,b]}(t)$ has 
units of energy/slot, and $P_{[n,b]}(t) \times (1\, \text{slot})$ is the resulting energy consumption in one slot. Also,  the energy harvested at time $t$ is assumed to be available for use in time $t+1$. }
\begin{eqnarray}
\sum_{b\in\script{N}^{(o)}_n}P_{[n, b]}(t) \leq E_n(t),\quad\forall\,\, n, \label{eq:energycond}
\end{eqnarray}
i.e., the consumed power must be no more than what is available. 

Each node in the network is assumed to be capable of harvesting energy from the environment, for instance, using solar panels \cite{opt-energy-twc10} or mechanical vibration \cite{vibrational-source-2014}. 
To capture the fact that the amount of harvestable energy typically varies over time, we use $h_n(t)$ to denote the amount of harvestable energy by node $n$ at time $t$, and  denote by $\bv{h}(t)=(h_1(t), ..., h_{N}(t))$ the harvestable energy vector at time $t$, called the \emph{energy state}.  We assume that $h_n(t)\leq h_{\max}$ for all $n, t$ for some finite $h_{\max}$. 
In the following, it is convenient for us to assume that each node can decide whether or not to harvest energy in each slot.  
%
Specifically, we use $e_n(t) \in[0, h_n(t)]$ to denote the amount of energy that is actually harvested at time $t$. We will see later that under our algorithm,  $e_n(t)\neq h_n(t)$ only when the energy storage is close to full. 
%

Denote $\bv{z}(t)=(\bv{S}(t), \bv{h}(t))$. We assume that $\bv{z}(t)$ takes values in $\script{Z}=\{\bv{z}_1, ..., \bv{z}_M\}$, where $\bv{z}_m=(\bv{s}_m, \bv{h}_m)$ and is i.i.d. every time. We denote $\pi_{m}=\prob{\bv{z}(t)=\bv{z}_m}$. We also rewrite $\script{P}^{(s_i)}$ as $\script{P}^{m}$ and $\mu_{[n, b]}(s_m, \bv{P})=\mu_{[n, b]}(\bv{z}_m, \bv{P})$. This allows arbitrary correlations among the harvestable energy processes and channels dynamics. \footnote{The i.i.d. assumption is made for ease of presentation. Our results can be extended to the case when $\bv{z}(t)$ evolves according to a general finite-state Markovian.}

\subsection{Queueing Dynamics}
Let $\bv{Q}(t)=(Q^{(c)}_n(t), n, c\in\script{N})$, $t=0, 1,2, ...$ be the data queue backlog vector in the network, where $Q^{(c)}_n(t)$ is  the amount of commodity $c$ data queued at node $n$. We assume the following queueing dynamics:
\begin{eqnarray}
Q^{(c)}_n(t+1) &\leq& \big[Q^{(c)}_n(t) - \sum_{b\in\script{N}^{(o)}_n}\mu_{[n, b]}^{(c)}(t)\big]^+ \label{eq:Qdynamic}\\
&&\qquad\qquad\,\,\,+ \sum_{a\in\script{N}^{(in)}_n}\mu_{[a, n]}^{(c)}(t)+R^{(c)}_n(t),\nonumber
\end{eqnarray}
with $Q^{(c)}_n(0)=0$ for all $n, c\in\script{N}$, $Q^{(c)}_c(t)=0$ $\forall\,t$, and $[x]^+=\max[x, 0]$. The inequality in (\ref{eq:Qdynamic}) is due to the fact that some nodes may not have enough commodity $c$ packets to fill the allocated rates. 
In this paper, we say that the network is \emph{stable} if the following condition is met:
\begin{eqnarray}
\overline{\bv{Q}} \triangleq \lim_{t\rightarrow\infty}\frac{1}{t}\sum_{\tau=0}^{t-1}\sum_{n,c} \expect{Q_n^{(c)}(\tau)} <\infty.\label{eq:queuestable}
\end{eqnarray}

Similarly, let $\bv{E}(t)=(E_n(t), n\in\script{N})$ be the vector of energy queue sizes.  Due to the energy availability constraint (\ref{eq:energycond}), we see that for each node $n$, the energy queue $E_n(t)$ evolves  according to the following: 
\begin{eqnarray}
\hspace{-.3in}&&E_n(t+1) =  E_n(t) - \sum_{b\in\script{N}^{(o)}_n}P_{[n, b]}(t) +e_n(t), \label{eq:Edynamic} 
\end{eqnarray}
with $E_n(0)=0$ for all $n$. 
Note that with (\ref{eq:Edynamic}), we start by  assuming that each energy queue has infinite capacity. We will show later that under our  algorithm, a finite buffer size is sufficient for achieving the desired perfromance. 


\subsection{Utility Maximization}
The goal of the network is to design a joint flow control, routing and scheduling, and energy  management algorithm to maximize the system utility, defined as: 
\begin{eqnarray}
U_{tot}(\overline{\bv{r}}) &=& \sum_{n, c} U^{(c)}_n(\overline{r}^{nc}),\label{eq:utiility-def}
\end{eqnarray}
subject to  network stability  (\ref{eq:queuestable}) and energy availability (\ref{eq:energycond}). Here $\overline{\bv{r}}=(\overline{r}^{nc}, \forall\, n, c\in\script{N})$ is the vector of the average expected admitted rates. We also use $\bv{r}^*$ to denote an optimal rate vector that maximizes (\ref{eq:utiility-def}) subject to (\ref{eq:queuestable}) and  (\ref{eq:energycond}). 

\subsection{Discussion  of the Model}
This model is general and can be used to model  systems that are self-powered and can harvest energy, e.g., environment monitoring wireless sensor networks,  or networks formed by mobile cellular devices. 
The same model was also considered in \cite{huangneely-energy-ton13}. There, two online algorithms were developed for achieving near-optimal utility performance. 
In this work, we use a  very different approach, which \emph{explicitly incorporates  learning}  into algorithm design and \emph{explores the benefits} of historic system information. Moreover, while previous works mostly focus on long term average performance, we also investigate the algorithm convergence time, defined to be the time it takes for the algorithm (and the system) to learn the optimal operating point.

\section{Algorithm Design via Learning} \label{section:alg-dual-learning}
In this section, we present our algorithm and the design approach. To facilitate understanding, we first discuss the intuition behind our approach. Then, we provide detailed descriptions of the algorithm. 

\subsection{Design Approach}
We first consider the following optimization problem, which can be intuitively viewed as the solution to our  problem.  \footnote{Technically speaking, one has to solve a ``convexified'' version of (\ref{eq:opt-obj}) to find an optimal control policy. But (\ref{eq:opt-obj}) is sufficient for our algorithm design and analysis. } 
\begin{eqnarray}
\hspace{-.2in} \max: && \phi = V\sum_{n, c} U^{(c)}_n(r^{nc})\label{eq:opt-obj}\\
\hspace{-.2in} \text{s.t.}  &&  r^{nc}+\sum_{m}\pi_{m} \sum_{a\in\script{N}^{(in)}_n}\mu^{(c)}_{[a, n]}(\bv{z}_m, \bv{P}^{m})\label{eq:opt-rate} \\
\hspace{-.2in}&&\qquad \leq\sum_{m}\pi_{m} \sum_{b\in\script{N}^{(o)}_n}\mu^{(c)}_{[n, b]}(\bv{z}_m, \bv{P}^{m}), \forall\,\,\, (n,c) \nonumber\\
\hspace{-.2in}&& \sum_{m}\pi_{m} \sum_{b\in\script{N}^{(o)}_n} P^{m}_{[n, b]}    =  \sum_{m}\pi_{m}  e_{n}^{m},\forall n \label{eq:opt-energy}  \\
\hspace{-.2in}&&  \bv{P}^{m}\in\script{P}_{m},  \forall \bv{z}_m,  \,  0\leq r^{nc} \leq R_{\max},\forall\, (n, c)\nonumber\\
\hspace{-.2in}&&  \sum_c\mu^{(c)}_{[n, b]}(\bv{z}_m, \bv{P}^{m})\leq \mu_{[n, b]}(\bv{z}_m, \bv{P}^{m}), \,\forall\, [n, b]   \nonumber\\
\hspace{-.2in}&&0\leq e_{n}^{m} \leq h_{n}^{m},\,\,\forall\,\, n, \bv{h}_j.\nonumber
\end{eqnarray}
Here $V\geq1$ is a constant and corresponds to a control parameter of our algorithm (explained later). 
Intuitively, problem (\ref{eq:opt-obj}) computes an optimal control policy. To see this, note that we can interpret $r^{nc}$ as the traffic admission rate, $\bv{P}^{m}$ as the power allocation vector under state $\bv{z}_m$, and $\bv{e}^{m}$ as the energy harvesting decision. (\ref{eq:opt-rate}) represents the queue stability constraint  and (\ref{eq:opt-energy}) denotes the energy consumption constraint. 

In practice,  one may not always have the statistics $(\pi_m, m)$ a-prior. As a result,  online algorithms have been proposed, e.g., ESA in \cite{huangneely-energy-ton13}, \cite{chen-energy-ton14}. However,  doing so \emph{ignores the historic system information} one can accumulate over time and loses its value. In our case, we try to explicitly utilize such information and to explore its benefits. 
Specifically, we will try to build an empirical distribution for the system dynamics $\{\bv{z}_m\}_{m=1}^M$. Then, we solve a \emph{perturbed empirical} version of the dual problem of (\ref{eq:opt-obj}) to  obtain an empirical Lagrange multiplier (we call this step \emph{perturbed dual learning}). After that, we incorporate the empirical multiplier into an online system controller  (Fig. \ref{fig:lem} shows its steps).

%

\vspace{-.05in}
\subsection{Learning-aided Energy Management}
Here we present our algorithm, which consists of an online controller and a learning component. 
We first present the algorithm and then explain  the controller in Section \ref{subsection:controller}. 

For our algorithm, we need the  dual problem of (\ref{eq:opt-obj}):  
\begin{eqnarray}
\min: \,\,\, g(\bv{\upsilon}, \bv{\nu}), \quad \text{s.t.} \,\,\, \bv{\upsilon}\succeq\bv{0}, \,\, \bv{\nu}\in\mathbb{R}^N,\label{eq:dual}
\end{eqnarray}
where $\bv{\upsilon}=(\upsilon^{(c)}_n, \forall\, (n, c))$, $\bv{\nu}=(\nu_n, \forall\,n)$ are the Lagrange multipliers, and the dual function $g(\bv{\upsilon}, \bv{\nu})\triangleq\sum_m\pi_mg_m(\bv{\upsilon}, \bv{\nu})$, where $g_m(\bv{\upsilon}, \bv{\nu})$ is defined as: 
\begin{eqnarray}
\hspace{-.35in}&&g_m(\bv{\upsilon}, \bv{\nu}) = \sup\bigg\{ V\sum_{n, c} U^{(c)}_n(r^{nc})  \nonumber\\
\hspace{-.35in}&&  \qquad - \sum_{n}\upsilon^{(c)}_n\big[r^{nc} +\sum_{a\in\script{N}^{(in)}_n} \mu^{(c)}_{[a, n]}(\bv{z}_m, \bv{P}^m)\label{eq:dual-function} \\
\hspace{-.3in}&&   \qquad - \sum_{b\in\script{N}^{(o)}_n}\mu^{(c)}_{[n, b]}(\bv{z}_m, \bv{P}^m)\big]    - \sum_{n}\nu_n\big[ e^m_n - \sum_{b\in\script{N}^{(o)}_n} P^m_{[n, b]} \big] \nonumber
\bigg\}.
\end{eqnarray}
Here the $\sup$ is taken over $\bv{r}$, $\bv{P}^m\in\script{P}^m$, $\bv{\mu}$, and $\bv{e}^m\preceq \bv{h}^m$. In the following, we use $(\bv{\upsilon}^*, \bv{\nu}^*)$   to represent an optimal solution of $g(\bv{\upsilon}, \bv{\nu})$.

We now present the algorithm, which uses a control parameter $V\geq1$ to tradeoff utility and delay, and specifies a learning time $T_L=V^c$ for some $c\in(0, 1)$. 
\begin{figure}[cht]
\vspace{-.1in}
\centering
\vspace{-.05in}
\includegraphics[width=2.8in, height=1.3in]{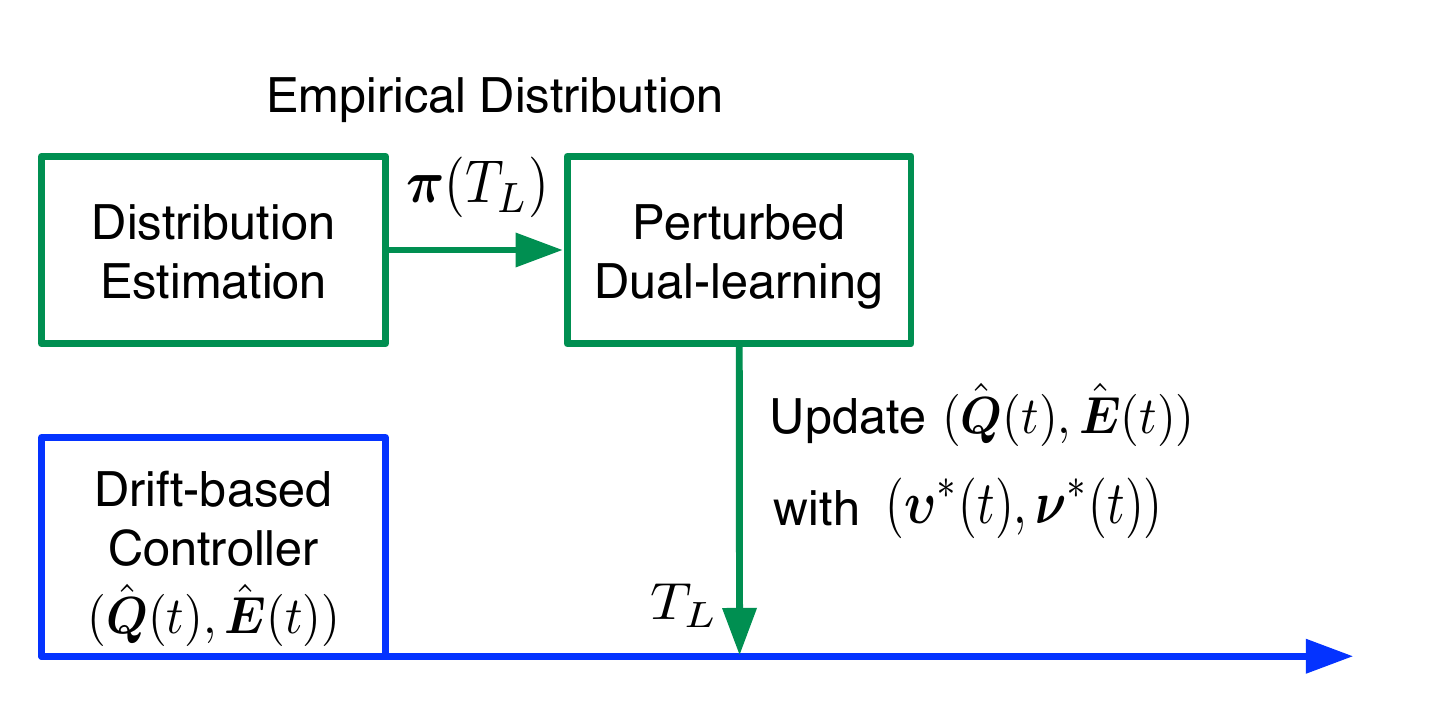}
\vspace{-.1in}
\caption{There are three main components in $\mathtt{LEM}$: (i) Build an empirical distribution $\bv{\pi}(t)$ for $\bv{z}(t)$.  (ii) Perform perturbed dual learning and obtain the empirical optimal multiplier at time $T_L$. (iii) Incorporate the multiplier into the controller.}
\label{fig:lem}
\vspace{-.08in}
\end{figure}


\underline{\textsf{Learning-aided Energy Management ($\mathtt{LEM}$):}} Initialize $\bv{\xi}_Q = \bv{0}$, $\bv{\xi}_E=\bv{0}$, and set $T_L=V^c$ with $c\in(0, 1)$. 
At every time $t$, observe $\bv{Q}(t)$, $\bv{E}(t)$,  $\bv{z}(t)$, and  define the following \emph{augmented queue vectors}: 
\begin{eqnarray}
\hat{\bv{Q}}(t) = \bv{Q}(t) + \bv{\xi}_Q, \quad \hat{\bv{E}}(t) = \bv{E}(t)+ \bv{\xi}_E. \label{eq:augmenting1}
\end{eqnarray}
Then, do: 
\begin{itemize}
\item \textbf{Energy harvesting:} If $\hat{E}_n(t)-\theta_n<0$, harvest energy,  i.e., set $e_n(t)=h_n(t)$. Else set $e_n(t)=0$. 
\item \textbf{Data admission:} For each $n$, choose $R^{(c)}_n(t)$ by solving the following optimization problem: 
\begin{eqnarray}
\hspace{-.2in}\max: \,\, VU^{(c)}_n(r) - \hat{Q}^{(c)}_n(t)r, \,\,\, s.t. \,\,0\leq r\leq R_{\max}. \label{eq:esa-admit}
\end{eqnarray}
\item  \textbf{Power allocation:} Define the weight of commodity $c$ data over link $[n, b]$ as:
\begin{eqnarray}
W^{(c)}_{[n, b]}(t) \triangleq \big[ \hat{Q}_n^{(c)}(t) - \hat{Q}_b^{(c)}(t) \big]^+. \label{eq:diff-backlog}
\end{eqnarray}
Then, define the link weight $W_{[n, b]}(t)=\max_{c}W^{(c)}_{[n, b]}(t)$, and  
choose $\bv{P}(t)\in\script{P}^{(\bv{z}(t))}$ to maximize: 
\begin{eqnarray}
\hspace{-.4in} &&G(\bv{P}(t))\triangleq\sum_{n}\bigg[\sum_{b\in\script{N}^{(o)}_n}\mu_{[n, b]}(t)W_{[n, b]}(t) \label{eq:esa-power}\\
\hspace{-.4in} &&\qquad\qquad\qquad\qquad+(\hat{E}_n(t)-\theta_n)\sum_{b\in\script{N}^{(o)}_n} P_{[n, b]}(t)\bigg]. \nonumber
\end{eqnarray}
\item \textbf{Routing and scheduling:} For every node $n$, find any $c^*\in\rmax_{c}W^{(c)}_{[n, b]}(t)$. If $W^{(c^*)}_{[n, b]}(t)>0$, set: 
\begin{eqnarray}
\mu_{[n, b]}^{(c^*)}(t) = \mu_{[n, b]}(t), \,\, \mu_{[n, b]}^{(c)}(t) = 0, \,\forall\,c\neq c^*. 
\end{eqnarray}
That is, allocate the full rate over link $[n, b]$ to any commodity that achieves the maximum positive weight over the link. Use idle-fill if needed.
\item \textbf{Queue update and packet dropping:} Use Last-In-First-Out (LIFO) for packet selection. If for any node $n$, the resulting $\{P_{[n, b]}(t), m\}$ in (\ref{eq:esa-power}) violates
constraint (\ref{eq:energycond}), set $\sum_{b\in\script{N}^{(o)}_n}P_{[n, b]}(t) = E_n(t)$ and drop all the packets that are supposed to be transmitted.   
Update $Q^{(c)}_n(t)$ and $E_n(t)$ according to (\ref{eq:Qdynamic}) and (\ref{eq:Edynamic}), respectively. 

\item \textbf{Perturbed Dual-learning at $T_L$:} Let $N_{m}$  be the number of times states $\bv{z}_m$ appear in $\{0, ..., T_L-1\}$. Denote $\pi_{m}(T_L)=\frac{N_{m}}{T_L}$  the empirical distribution of $\bv{z}_m$. Solve: 
\begin{eqnarray}
\hspace{-.15in}\min: \,\,\, \sum_m\pi_m(T_L)g_m(\bv{\upsilon}, \bv{\nu}-\bv{\theta}), \,\, \text{s.t.} \,\,\, \bv{\upsilon}\succeq\bv{0},  \bv{\nu}\in\mathbb{R}^N, \label{eq:dual-emprical}
\end{eqnarray}
and obtain the optimal multiplier $(\bv{\upsilon}^*(T_L), \bv{\nu}^*(T_L))$. Change  $\bv{\xi}_Q$ and $\bv{\xi}_E$ in (\ref{eq:augmenting1}) to: 
\begin{eqnarray}
\bv{\xi}_Q &=& \bv{\upsilon}^*(T_L) -  V^{1-\frac{c}{2}}\log(V)^2\cdot\bv{1}  \label{eq:augmenting2} \\
\bv{\xi}_E &=&\bv{\nu}^*(T_L) -  V^{1-\frac{c}{2}}\log(V)^2\cdot\bv{1}. \,\,\Diamond \label{eq:augmenting3}
\end{eqnarray}
\end{itemize}

We will explain the controller in the next subsection. Here, we first 
note that the perturbed dual learning step is performed \emph{only once} at time $t=T_L$. \footnote{One can also devise a version of $\mathtt{LEM}$ which does continuous learning. }  Also, although $\mathtt{LEM}$ is equipped with a packet dropping option to ensure zero energy outage,  dropping rarely happens, i.e., $O(V^{-\log(V)})$. Moreover, we will show that the energy availability constraint is always ensured. 

\vspace{-.06in}
\subsection{Remarks on LEM}
$\mathtt{LEM}$ only requires  knowledge of the \emph{instantaneous}  state $\bv{z}(t)$ and queue states $\bv{Q}(t)$ and $\bv{E}(t)$. 
\emph{It does not  require any statistical information about $\bv{S}(t)$ or any knowledge of the energy state process $\bv{h}(t)$.}   
This is a very useful feature, as  exact knowledge of the energy source may be difficult to obtain at the beginning. 

There is an explicit learning step in $\mathtt{LEM}$. This distinguishes it  from previous  algorithms for energy harvesting networks, e.g., \cite{chen-energy-ton14}, \cite{huang-mobile-wiopt13},  \cite{power-harvesting-kansal07}, where sufficient statistical knowledge of the energy source is often required and no learning is considered.  We will show in Theorem \ref{theorem:lem-conv} that  $\mathtt{LEM}$  converges in $O(V^{2/3})$ time (up to a $\log$ factor), which is much faster than the  $\Theta(V)$ time for algorithms based purely on queues, or the $\Theta(V^2)$ time for algorithms based purely on learning the statistics. 

The  perturbation approach here is needed for guaranteeing the feasibility of dual learning and the resulting algorithm. Specifically, it ``shifts'' the optimal Lagrange multiplier to a positive value via $\bv{\theta}$. This step allows us to track the negative multiplier with positive queue sizes for decision making, and is critical for networks with the ``no-underflow'' constraint, e.g.,  processing networks \cite{jiang-spn}. 

Finally, note that due the general rate functions $\{\mu_{[nm]}(\bv{z}, \bv{P})\}$, our problem inevitably requires a centralized controller for achieving optimality. Thus, $\mathtt{LEM}$ also requires centralized implementation. 
In the special case when network links do not interfere with each other and the dynamics are all independent,  nodes can estimate the local distributions and pass the information to a leader node to compute $(\bv{\upsilon}^*(T_L), \bv{\nu}^*(T_L))$. Then, the leader node sends back the multiplier information to the nodes. After that,  $\mathtt{LEM}$ can be implemented  in a distributed manner. 

%
%




%

%

\vspace{-.04in}
\subsection{Information Augmented Controller}\label{subsection:controller}
Here we provide  mathematical explanations for our controller. As we will see, the control rules are results of a  drift minimization principle  \cite{neelynowbook}, \emph{augmented by the information learned in perturbed dual learning}. 


To start, we define a \emph{perturbed} Lyapunov function as follows:
\begin{eqnarray}
L(t)\triangleq\frac{1}{2}\sum_{n, c\in\script{N}} \big[Q^{(c)}_n(t)\big]^2+\frac{1}{2}\sum_{n\in\script{N}} \big[E_n(t)-\theta_n\big]^2.\label{eq:lyapunov-func}
\end{eqnarray}
Denote $\bv{Y}(t)=(\bv{Q}(t), \bv{E}(t))$ and define a one-slot conditional Lyapunov drift as follows:
\begin{eqnarray}
\Delta(t) \triangleq \expect{L(t+1) - L(t)\left.|\right. \bv{Y}(t)}. 
\end{eqnarray}
We then have the following lemma from \cite{huangneely-energy-ton13}. 
\begin{lemma}\label{lemma:drift}
Under any feasible data admission action, power allocation action that satisfies constraint (\ref{eq:energycond}), routing and scheduling action, and energy harvesting action that can be implemented at time $t$, we have: 
\begin{eqnarray}
\hspace{-.3in} &&\Delta(t) -V\expect{\sum_{n, c} U^{(c)}_n(R^{(c)}_n(t))\left.|\right. \bv{Y}(t)} \label{eq:drift2}\\
\hspace{-.3in} &&\leq B +\sum_{n\in\script{N}}(E_n(t)-\theta_n)\expect{e_n(t)\left.|\right. \bv{Y}(t)}  \nonumber\\
\hspace{-.3in} && \qquad\,  - \expect{ \sum_{n, c} \big[VU^{(c)}_n(R^{(c)}_n(t)) - Q^{(c)}_n(t)R^{(c)}_n(t)\big] \left.|\right. \bv{Y}(t)} \nonumber\\ 
\hspace{-.3in} && \qquad\, - \expect{\sum_{n} \bigg[\sum_{c}\sum_{b\in\script{N}^{(o)}_n}\mu_{[n, b]}^{(c)}(t) \big[ Q_n^{(c)}(t) - Q_b^{(c)}(t)\big] \nonumber\\
\hspace{-.3in} && \qquad \qquad \quad \quad \quad\,\, +(E_n(t)-\theta_n)\sum_{b\in\script{N}^{(o)}_n} P_{[n, b]}(t) \bigg]\left.|\right.\bv{Y}(t)}.  \nonumber
\end{eqnarray}
%
Here $B\triangleq N^2(\frac{3}{2}d^2_{\max}\mu^2_{\max}+R_{\max}^2)+\frac{N}{2}(P_{\max}+h_{\max})^2$, and $d_{\max}$ is defined as the maximum in-degree/out-degree of any node in the network. $\Diamond$
\end{lemma}
\begin{proof}
See \cite{huangneely-energy-ton13}. 
\end{proof}

Now add to both sides of (\ref{eq:drift2}) the following \emph{drift-augmenting} term, which carries the information  learned in the dual learning step, i.e.,  $\bv{\xi}_Q$ and $\bv{\xi}_E$: 
\begin{eqnarray}
\hspace{-.4in}&&\Delta_A(t)\triangleq  - \expect{\sum_n\xi_{Q,n}^{(c)} [  \sum_{b\in\script{N}^{(o)}_n}\mu_{[n, b]}^{(c)}(t) \label{eq:augmenting}\\
\hspace{-.4in}&&\qquad\qquad\qquad\quad\quad\quad - \sum_{b\in\script{N}^{(in)}_n}\mu_{[a, n]}^{(c)}(t)  - R_n^{(c)}(t)  ]   \left.|\right. \bv{Y}(t)}\nonumber\\
\hspace{-.4in}&&\qquad \quad\quad - \expect{\sum_n\xi_{E, n} [ \sum_{b\in\script{N}^{(o)}_n} P_{[n, b]}(t) - e_n(t) ]   \left.|\right. \bv{Y}(t)}. \nonumber
\end{eqnarray}
 Doing so, one obtains the following \emph{augmented drift}: 
\begin{eqnarray}
\hspace{-.3in} &&\Delta(t) + \Delta_A(t)-V\expect{\sum_{n, c} U^{(c)}_n(R^{(c)}_n(t))\left.|\right. \bv{Y}(t)} \label{eq:drift3}\\
\hspace{-.3in} &&\leq B +\sum_{n\in\script{N}}(\hat{E}_n(t)-\theta_n)\expect{e_n(t)\left.|\right. \bv{Y}(t)}  \nonumber\\
\hspace{-.3in} && \qquad   - \expect{ \sum_{n, c} \big[VU^{(c)}_n(R^{(c)}_n(t)) - \hat{Q}^{(c)}_n(t)R^{(c)}_n(t)\big] \left.|\right. \bv{Y}(t)} \nonumber\\ 
\hspace{-.3in} &&  \qquad - \expect{\sum_{n} \bigg[\sum_{c}\sum_{b\in\script{N}^{(o)}_n}\mu_{[n, b]}^{(c)}(t) \big[ \hat{Q}_n^{(c)}(t) - \hat{Q}_b^{(c)}(t)\big] \nonumber\\
\hspace{-.3in} && \qquad \qquad \quad \quad \quad  +(\hat{E}_n(t)-\theta_n)\sum_{b\in\script{N}^{(o)}_n} P_{[n, b]}(t) \bigg]\left.|\right.\bv{Y}(t)}.  \nonumber
\end{eqnarray}
Comparing (\ref{eq:drift3}) and $\mathtt{LEM}$, we see that $\mathtt{LEM}$   is constructed to \emph{minimize the right-hand-side (RHS) of the augmented drift (\ref{eq:drift3})}. This augmenting step is important and provides a way to incorporate  learning into control algorithm design.

\section{Performance analysis}\label{section:analysis}

Here we present the performance results  for $\mathtt{LEM}$. We first state the assumptions. Then, we present the theorems.   

\subsection{Assumptions}
In our analysis, we make the following assumptions. 
\begin{assumption}\label{assumption:bdd-LM}  
There exists a constant $\epsilon=\Theta(1)>0$ such that for any valid distributions $\hat{\bv{\pi}} = (\hat{\pi}_{1}, ..., \hat{\pi}_{M})$  with $\|  \hat{\bv{\pi}} - \bv{\pi} \|\leq \epsilon$, there exist a set of actions $\{R_{n, k}^{(c)}\}_{k\in\mathbb{N}_+}$, $\{\bv{P}_{k}^{m}\}^{m}_{k\in\mathbb{N}_+}$,  $\{\bv{\mu}_{k}^{m}\}^m_{k\in\mathbb{N}_+}$,  and $\{ \bv{e}_{k}^{m}\}^m_{k\in\mathbb{N}_+}$, and distributions $\{\vartheta^{m}_k\}^m_{k\in\mathbb{N}_+}$, and $\{\varrho^{m}_{k}\}^m_{k\in\mathbb{N}_+}$  (possibly dependent on $\hat{\bv{\pi}}$), such that (i) there exists $\eta_0=\Theta(1)>0$ independent of $\hat{\bv{\pi}}$, so that: 
\begin{eqnarray}
\hspace{-.3in}&&\sum_{m}\pi_{m}\big\{\sum_k\vartheta^{m}_k[ R_{n, k}^{(c)} +  \sum_{a\in\script{N}^{(in)}_n}\mu_{[a, n], k}^{(c)},  \label{eq:slackness1} \\
\hspace{-.3in}&& \qquad\qquad\qquad \qquad  -   \sum_{b\in\script{N}^{(o)}_n}\mu_{[n, b], k}^{(c)} ]\big\} \leq  -\eta_0, \,\,\forall\,\, n, c,\nonumber 
\end{eqnarray}
and for each $n$, 
\begin{eqnarray}
\hspace{-.5in}&&\sum_m\pi_m\sum_k\varrho^{m}_{k} e_{n,k}^{m} - \sum_{m}\pi_{m}\sum_k\vartheta^{m}_k\sum_{b\in\script{N}^{(o)}_n} P_{[n, b], k}^{m}  =  0,  \label{eq:slackness2}
\end{eqnarray}
and  (ii)  $0<  \sum_m\pi_m\sum_k\varrho^{m}_{k} e_{n,k}^{m}< \sum_m\pi_mh^m_n$ $\forall\, n$. $\Diamond$
\end{assumption}

Although Assumption \ref{assumption:bdd-LM} appears complicated, it indeed only assumes that the system has a ``slackness'' property, so that there exists a stationary and randomized policy that can stabilize the system, and the resulting service rates are slightly larger than the arrival rates for the queues. 
Assumption \ref{assumption:bdd-LM} is a necessary condition for achieving network stability and  is often assumed in network optimization works with $\epsilon=0$, e.g., \cite{eryilmaz_qbsc_ton07}.  
Here with $\epsilon>0$, we assume that systems with slightly different channel and harvestable energy distributions can also be stabilized with the same slack (the stabilizing policy may be different). 

\subsection{Performance Results}
Here we present the performance results.  We first define the following structural property of the system,  which will be used in our analysis. 
\begin{definition}
A system is called \emph{polyhedral} with parameter $\rho>0$, if the dual function $g(\bv{v}, \bv{\nu})$ satisfies:
\begin{eqnarray}
g(\bv{v}^*, \bv{\nu}^*)\leq g(\bv{v}, \bv{\nu})-\rho\|(\bv{v}^*, \bv{\nu}^*)-(\bv{v}, \bv{\nu})\|.\,\,\,\Diamond\label{eq:polyhedral}
\end{eqnarray}
\end{definition}
This polyhedral property  often appears in practical systems, especially when the control actions are discrete (see \cite{huangneely_dr_tac} for more discussions). Moreover, (\ref{eq:polyhedral}) holds for all $V$ values whenever it holds for $V=1$.

Our first  lemma shows that with $\theta_n=V\log(V)$, one can guarantee that at time $T_L$, the empirical multipliers $\bv{v}^*(t)$ and $\bv{\nu}^*(t)$ are close to their true values with high probability. 
\begin{lemma}\label{lemma:beta-conv}
For a sufficiently large $V$, with probability $1- O(\frac{1}{V^{4\log(V)}})$, at time $t=T_L=V^c$ with $c\in(0, 1)$, one has: 
\begin{eqnarray}
\hspace{-.4in}&&\| (\bv{v}^*(t), \bv{\nu}^*(t))  - (\bv{v}^*, \bv{\nu}^*+\bv{\theta}) \| = \Theta(V^{1-\frac{c}{2}}\log(V)). \label{eq:approx-guess} 
\end{eqnarray}
Here $\bv{\nu}^*+\bv{\theta} = \Theta(V\log(V))>0$. $\Diamond$
\end{lemma}
\begin{proof}
See Appendix A. 
\end{proof}
Since $\bv{\nu}^*+\bv{\theta} = \Theta(V\log(V))$, we see that the relative error of $(\bv{v}^*(t), \bv{\nu}^*(t))$ is quite small. 
This high accuracy (with respect to the size of $(\bv{v}^*, \bv{\nu}^*+\bv{\theta})$)  contributes to  achieving a good performance and fast convergence rate for $\mathtt{LEM}$. 
Here $\bv{\nu}^*+\bv{\theta} = \Theta(V\log(V))>0$ is important,  because  without $\bv{\theta}$, we may get a non-positive $\bv{\nu}^*$ after solving (\ref{eq:dual-emprical}), due to the fact that (\ref{eq:opt-energy}) is an equality constraint. In that  case, it is impossible to use $\bv{E}(t)$ to track  $\bv{\nu}^*$ and to base decisions on $\bv{E}(t)$. 

We  now state our first main theorem, which  summarizes the performance of $\mathtt{LEM}$. 
\begin{theorem}\label{theorem:lem}
Suppose  that the dual function $g(\bv{v}, \bv{\nu})$ is polyhedral with $\rho = \Theta(1)>0$, i.e., independent of $V$, and has a unique optimal $(\bv{v}^*, \bv{\nu}^*)$ with $\bv{v}^*>0$. Then, under $\mathtt{LEM}$  with $\theta_n=V\log(V)$ and a sufficiently large $V$,   with probability $1- O(\frac{1}{V^{4\log(V)}})$, we have: 
\begin{enumerate}
\item[(a)] 
The average  queue sizes satisfy:  
\begin{eqnarray}
\hspace{-.3in}&& Q^{(c)}_{n, \text{av}} \leq \frac{3}{2}V^{1-\frac{c}{2}}\log(V)^2 + O(1),\,\,\,\forall\,\,(n, c), \label{eq:data-queue-bound}\\
\hspace{-.3in}&& E_{n, \text{av}} \leq \frac{3}{2}V^{1-\frac{c}{2}}\log(V)^2 + O(1), \,\,\, \forall\,\, n.\label{eq:energy-queue-bound}
\end{eqnarray}
In particular, in steady state, there exist $\Theta(1)$ constants $D, \xi, K$ such that: 
\begin{eqnarray}
\hspace{-.6in}&&\prob{ Q_n^{(c)}(t) \geq \frac{3}{2}V^{1-\frac{c}{2}}\log(V)^2+D +b}  \leq \xi e^{-Kb}\label{eq:queue-prob-bdd} \\
\hspace{-.6in}&&\prob{ E_n(t) \geq \frac{3}{2}V^{1-\frac{c}{2}}\log(V)^2+D+b} \leq \xi e^{-Kb}.  \label{eq:energy-prob-bdd}
\end{eqnarray}
\item[(b)] For every data queue $j$ with arrival rate $\lambda_j>0$, there exist a set of packets with rate $\tilde{\lambda}_j\geq [\lambda_j - O(1/V^{\log(V)})]^+$, such that their average delay at queue $j$ is $O(\log(V)^2)$. 
\item[(c)] Let $\bv{\overline{r}}=(\overline{r}^{nc}, \forall\, (n, c))$ be the time average admitted rate vector achieved by $\mathtt{LEM}$ defined in Section \ref{subsection:rate-utility}. We have:  
\begin{eqnarray}
U_{tot}(\overline{\bv{r}}) &\geq& U_{tot}(\bv{r}^*) -  O(\frac{1}{V}). \label{eq:utility-perf}
\end{eqnarray}
Here $\bv{r}^*$ is an optimal solution of our problem. Moreover,  no dropping takes places before time $T_L$ and the average packet dropping rate is $O(1/V^{\log(V)})$.  $\Diamond$ 
\end{enumerate}
\end{theorem}
\begin{proof}
See Appendix  B. 
\end{proof}

By taking $\epsilon=1/V$, we see from Part (a) and Part (b) that $\mathtt{LEM}$ achieves an $[O(\epsilon), O( \log(\frac{1}{\epsilon})^2)]$ utility-delay tradeoff. 
We also see from Part (a) that $\mathtt{LEM}$ can  use an energy buffer of size  $O((\frac{1}{\epsilon})^{1-\frac{c}{2}}\log(\frac{1}{\epsilon})^2)$, which is much smaller than the $\Theta(1/\epsilon)$ size under previous algorithms. 
%


Our second main result concerns the \emph{convergence time} of $\mathtt{LEM}$. The convergence time of an algorithm characterizes how fast it (or equivalently, the system) enters its steady state. 
A faster convergence speed implies faster learning and more efficient resource allocation. 
The formal definition of the convergence time is as follows \cite{huang-learning-sig-14}: 
\begin{definition}
Let $\zeta>0$ be a given constant. The convergence time of the control algorithm is defined as: 
\begin{eqnarray}
T_\zeta\triangleq \inf\{t:  \| (\hat{\bv{Q}}(t), \hat{\bv{E}}(t))  - (\bv{v}^*, \bv{\nu}^*+\bv{\theta})  \|<\zeta  \}. \Diamond
\end{eqnarray}
\end{definition}
Here the intuition is that once $(\hat{\bv{Q}}(t), \hat{\bv{E}}(t))$ gets close to $(\bv{v}^*, \bv{\nu}^*+\bv{\theta})$, $\mathtt{LEM}$ will start making near-optimal decisions. 
\begin{theorem}\label{theorem:lem-conv}
Suppose the conditions in Theorem \ref{theorem:lem} hold. 
Under $\mathtt{LEM}$, there exists an $\Theta(1)$ constant $D$ such that, with probability $1- O(\frac{1}{V^{4\log(V)}})$, 
\begin{eqnarray}
\expect{T_{D}} = O(V^c + V^{1-\frac{c}{2}}\log(V)^2). 
\end{eqnarray}
In particular, when $c=\frac{2}{3}$,  $\expect{T_D} = O(V^{\frac{2}{3}}\log(V)^2)$. $\Diamond$
\end{theorem}
\begin{proof}
See Appendix C. 
\end{proof}
We remark here that if one only uses pure queue-based policies to track the optimal multipliers, e.g., ESA in \cite{huangneely-energy-ton13} (can be viewed as linear learning since each queue can change by an $\Theta(1)$ amount at each time),  the convergence time is necessarily $\Theta(V)$, since the optimal multiplier is $\Theta(V)$ \cite{huangneely_dr_tac}. If instead one tries to compute the optimal solution only  by learning the distribution, it requires $\Theta(V^2)$ time to ensure that the distribution is within $O(1/V)$ accuracy. 
%
Dual learning can be viewed as combining the benefits of the two methods, i.e., the fast start of statistical learning and the smooth learning of queue-based policies. Hence, it is able to achieve a superior  convergence speed compared to both methods.

\vspace{-.06in}
\section{Simulation}\label{section:sim}
This section provides simulation results for $\mathtt{LEM}$. We consider the network shown in Fig. \ref{fig:simtopo}, which is an example of a data collecting sensor network. 
In this network, traffic are admitted from nodes $1$, $2$ and $3$, and are relayed to  node $4$.  Since we only have one commodity, we omit the superscript. 
\begin{figure}[cht]
\centering
\vspace{-.1in}
\includegraphics[height=1.0in, width=2.0in]{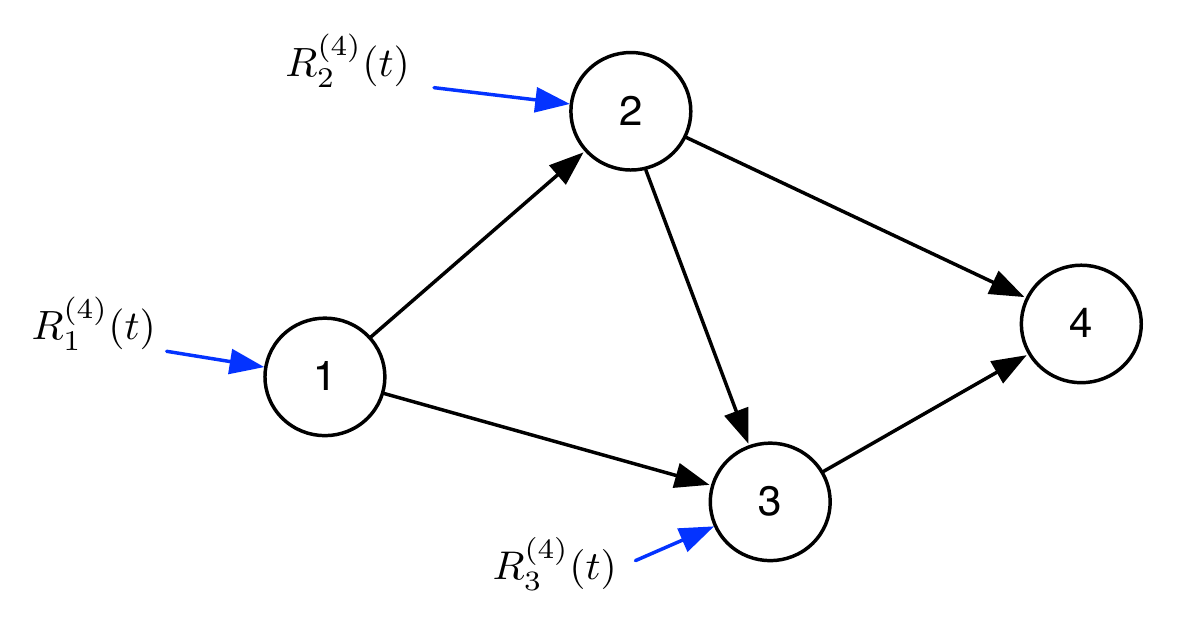}
\vspace{-.1in}
\caption{A data collection network. }\label{fig:simtopo}
\vspace{-.1in}
\end{figure}

The channel state of each communication link  is i.i.d. every time slot and can be either ``G=Good'' or ``B=Bad.'' The probabilities of being in the good state for the links are given by $\bv{p}^s = (p^s_{12}, p^s_{13}, p^s_{23}, p^s_{24}, p^s_{34}) = (0.5, 0.2, 0.3, 0.5,0.7)$. 
For each node,  the harvestable energy $h_n(t)$ takes values $2$ or $0$. The probabilities of having a $2$-unit energy arrival at the nodes are $\bv{p}^h=(p^h_{1}, p^h_{2}, p^h_{3})=(0.6, 0.3, 0.5)$. 
Note that we have a total of $32$ channel states and $8$ energy states. 

%
At every time $t$, a node can either allocate one unit power for transmission or do not transmit. 
When the channel is good, one unit power can support a transmission of two packets. Otherwise it can only support one. We assume $R_{\max}=2$ and each time $R_n(t)\in\{0, 1, 2\}$. 
The utility functions are given by: $U_{1}(r)=3\log(1+r) $ and $U_{2}(r)=2\log(1+r)$, and $U_{3}(r)=\log(1+r)$. We also assume that the links do not interfere with each other. 

%
We simulate $\mathtt{LEM}$ with $V\in\{30, 40, 50, 80, 100, 150\}$ and $c=2/3$. We choose to begin with $V=30$ so that dropping does not happen.  Each simulation is run for $10^6$ slots.  In the simulation, in order to combact the effect of $V$ not being large enough, we slightly increase the learning time from $V^{c}$ to $V^{c}\log(V)$ (same performance can be proven). We also reduce  $V^{1-c/2}\log(V)^2$ in (\ref{eq:augmenting2}) and (\ref{eq:augmenting3}) to $V^{1-c/2}\log(V)$, the results are not affected. 
For benchmark comparison, we also  implement the ESA algorithm  in \cite{huangneely-energy-ton13}. \footnote{Other algorithms in the literature are designed for different settings and do not directly apply to our problem. }
\begin{figure}[cht]
\centering
\includegraphics[height=2in, width=3.3in]{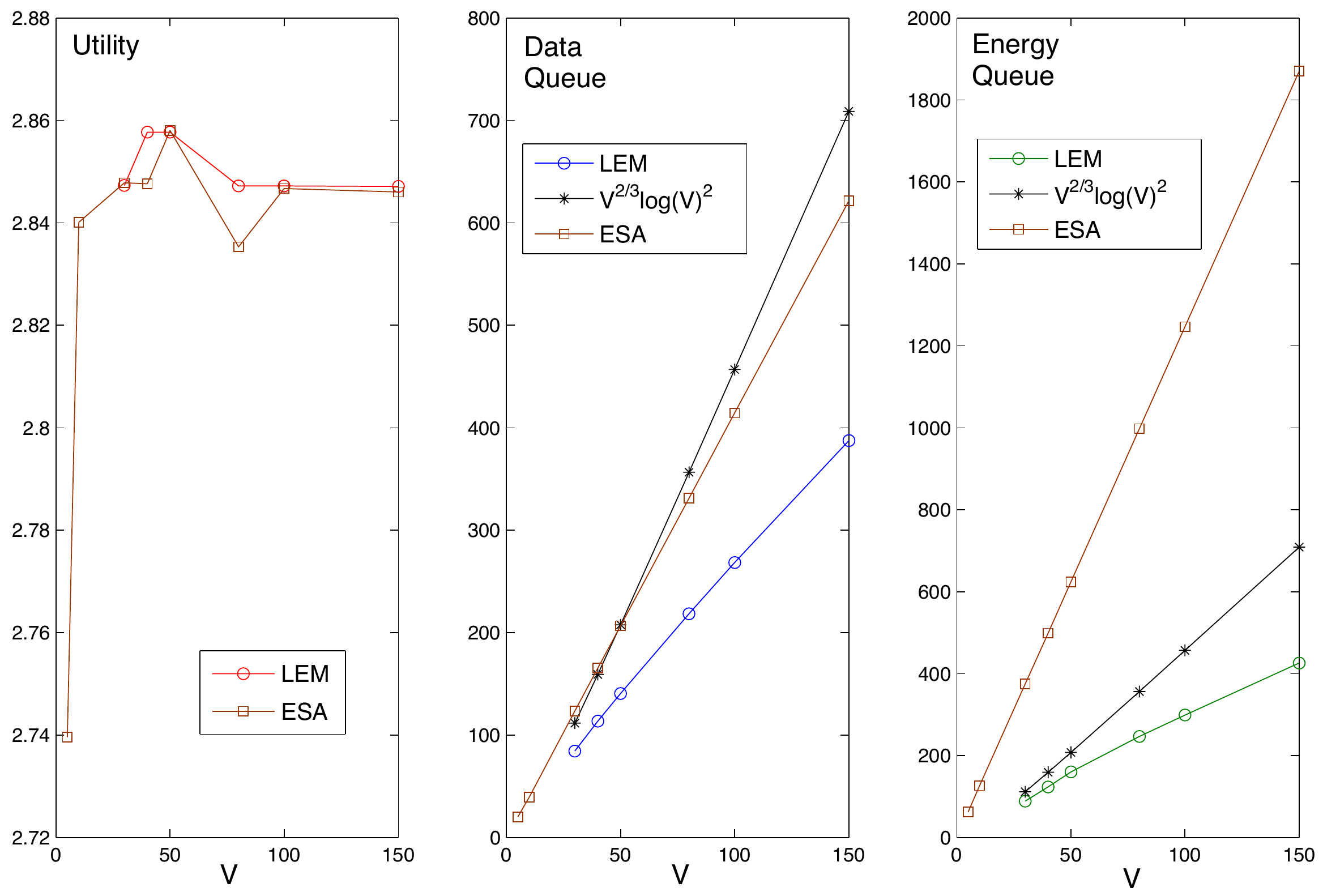}
\vspace{-.15in}
\caption{Utility and queue performance of LEM. }\label{fig:lem-per}
\vspace{-.1in}
\end{figure}

Fig. \ref{fig:lem-per} shows the utility and queue performance of $\mathtt{LEM}$. We see that  $\mathtt{LEM}$ achieves good utility performance. Here the small improvements for different $V$ values under $\mathtt{LEM}$ are because at $V=30$, the utility performance is already close to optimal. 
We also see from the middle and the right plots that both the average data queues and the average energy queues under $\mathtt{LEM}$ are of size $O(V^{1-c/2}\log(V)^2)$, whereas it is $\Theta(V)$ under ESA.  This implies that one can implement $\mathtt{LEM}$ with much smaller energy buffers  compared to ESA. 
\begin{figure}[cht]
\centering
\includegraphics[height=1.6in, width=2.8in]{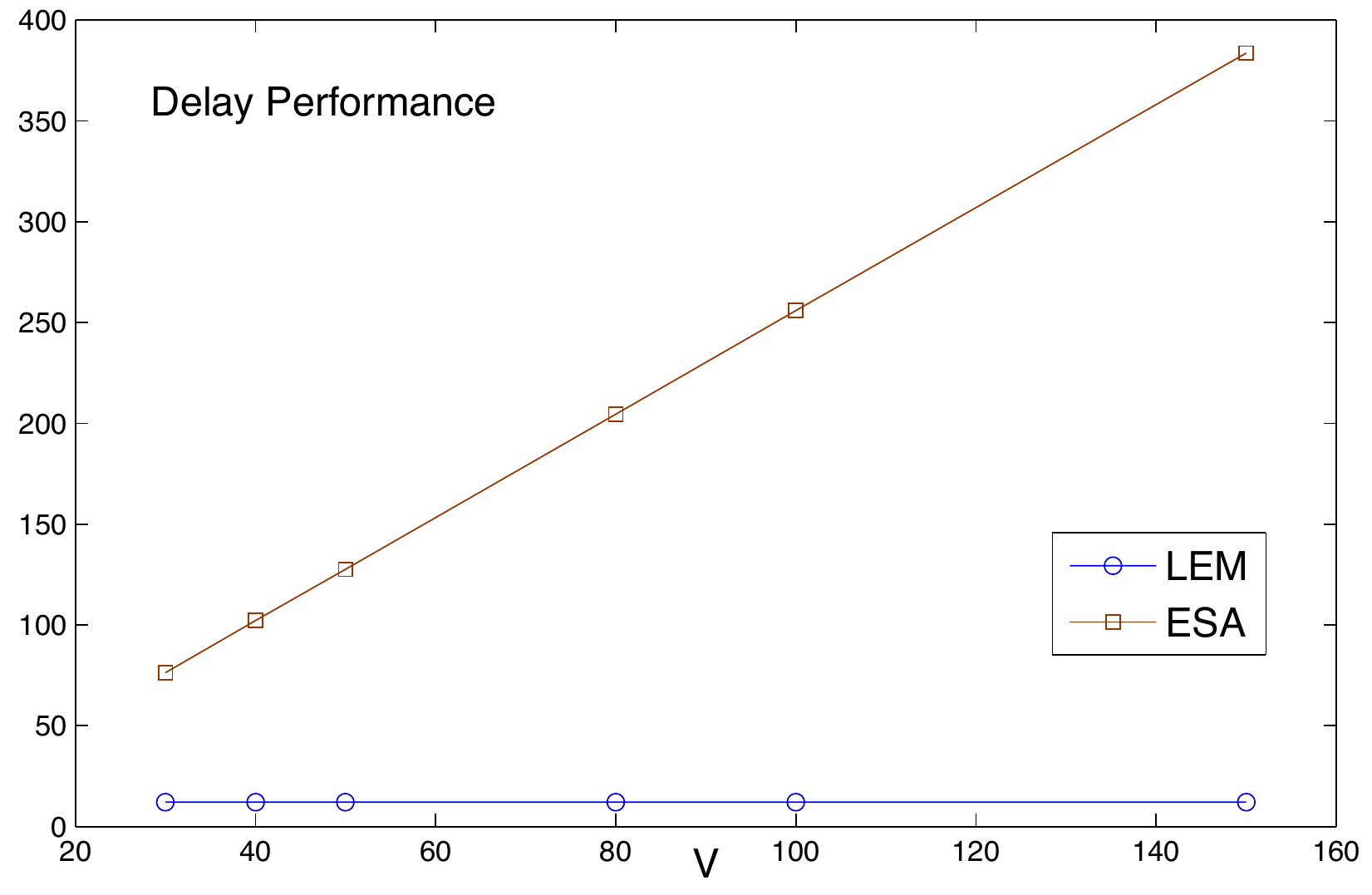}
\vspace{-.15in}
\caption{Delay scaling under LEM. }\label{fig:lem-delay}
\vspace{-.1in}
\end{figure}

Fig. \ref{fig:lem-delay} also shows the average packet delay under $\mathtt{LEM}$, computed using the packets that   exit  the network when the simulation ends (This accounts for more than $99.9\%$ of the packets that enter the network). 
It can be seen that the average packet delay under $\mathtt{LEM}$ grows very slowly, i.e., $O(\log(V)^2)$, and stays around $12$ slots. On the other hand, the delay under ESA grows linearly in $V$, and ranges from $75$ to $380$ slots (ESA requires a $V\geq30$ to achieve a similar utility performance as $\mathtt{LEM}$).  Thus, with the same utility performance, $\mathtt{LEM}$ achieves a $6$ to $30$-fold delay saving. 
\begin{figure}[cht]
\centering
\vspace{-.1in}
\includegraphics[height=2in, width=3.3in]{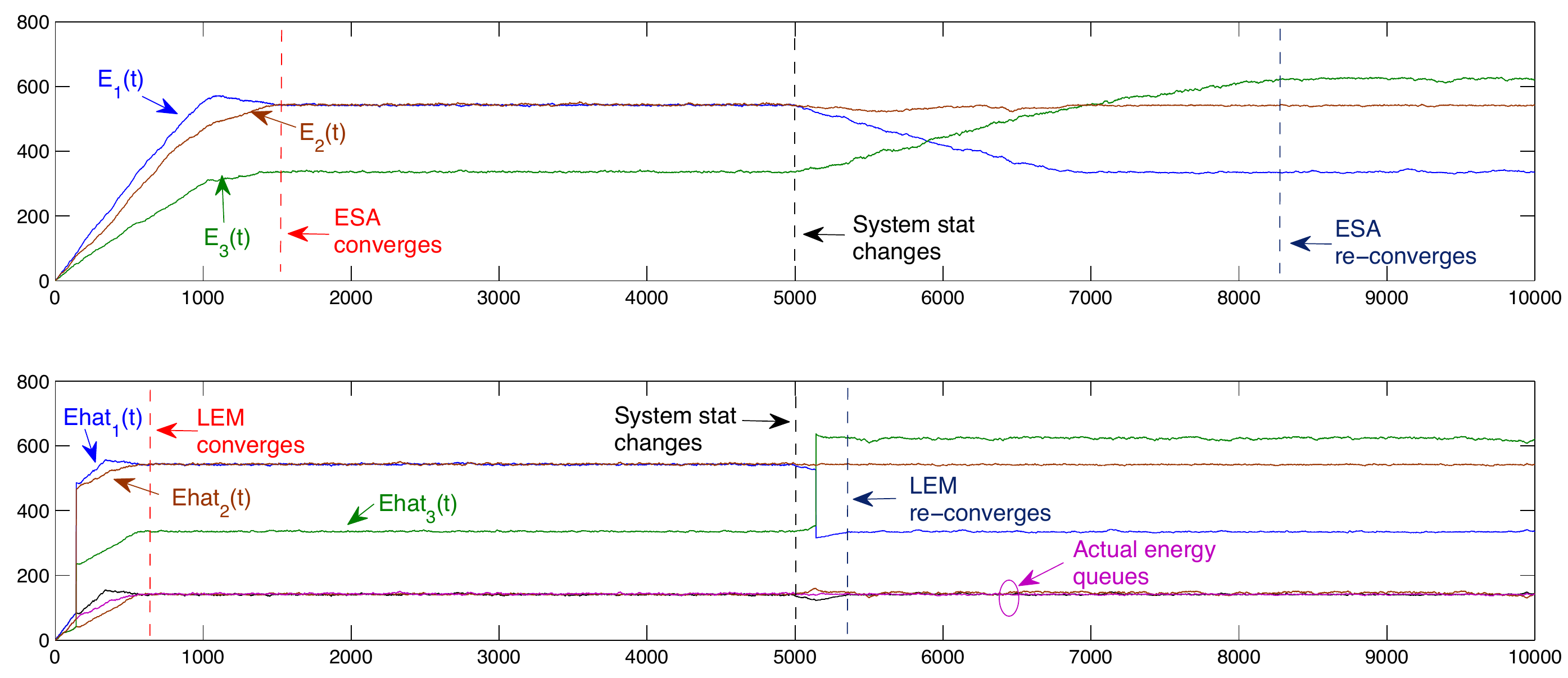}
\vspace{-.15in}
\caption{Convergence comparison between $\mathtt{LEM}$ and  ESA for $V=150$ }\label{fig:conv-per}
\vspace{-.1in}
\end{figure}

Fig. \ref{fig:conv-per} shows the convergence property of $\mathtt{LEM}$. To show the convergence, we shorten the time to $10^4$ slots and  change the system statistics in slot $5000$ to  $\bv{p}^s = (0.3, 0.2, 0.2, 0.5,0.7)$ and $\bv{p}^h=(0.1, 0.6, 0.2)$.  
We see that $\mathtt{LEM}$ converges much faster compared to ESA (We only show $\bv{E}(t)$ here. The data queues have similar performance). 
Indeed, in the first time,  the energy queue sizes converge to the optimal values (the corresponding optimal Lagrange multiplier values) at around $650$ slots under $\mathtt{LEM}$, whereas ESA converges at around $1500$ slots (an $850$-slot saving!).  Then, after we apply the change, $\mathtt{LEM}$ re-converge after about $450$ slots, whereas ESA takes about $3300$ slots to re-adapt to the system ($7\times$ faster, save about $2900$ slots!). Moreover, we also note that the actual energy queue sizes are barely affected, except for a small change after time $5000$. 
This clear demonstrates the effectiveness of using dual learning in accelerating the convergence of the algorithm. 





\section{Conclusion} \label{section:con}
In this paper, we develop a learning-aided energy management algorithm ($\mathtt{LEM}$) for general multihop energy harvesting networks. $\mathtt{LEM}$ explicitly utilizes  historic system information and learns an empirical optimal Lagrange multiplier  via perturbed dual learning. Then, it incorporates the multiplier into a drift-based system controller via drift-augmenting. We show that $\mathtt{LEM}$ is able to achieve a near-optimal utility-delay tradeoff with a finite energy storage capacity. Moreover, $\mathtt{LEM}$ possesses a provable faster convergence speed compared to existing techniques that based purely on queue-based control or based purely on learning the statistics.  
 
\bibliographystyle{unsrt}
\bibliography{mybib}

\vspace{-.1in}
 
\section*{Appendix A -- Proof of Lemma \ref{lemma:beta-conv}}
Here we prove Lemma \ref{lemma:beta-conv}. 
\begin{proof} (Lemma \ref{lemma:beta-conv}) 
First of all, the result $\bv{\nu}^*+\bv{\theta} = \Theta(V\log(V))>0$ follows from Lemma 1 in \cite{huangneely_dr_tac}, which states that $\bv{\nu}^*=\Theta(V)$. 

Now we see that $\pi_{m}(t)$ will converge to $\pi_{m}$ with probability $1$ \cite{durrett_prob}. Therefore, there exists an $\Theta(1)$ time $T_{\epsilon}$, such that $\|\bv{\pi}(t) - \bv{\pi} \|\leq \epsilon$ for all $t\geq T_{\epsilon}$ with probability $1$.  

Then, at every time $t\geq T_{\epsilon}$, we use Theorem 1 in \cite{huangneely_qlamarkovian} to obtain that: 
\begin{eqnarray}
&&g(\bv{\upsilon}^*(t), \bv{\nu}^*(t)-\bv{\theta}, t) = \phi^*_c(t) \leq VU_{\max}, 
\end{eqnarray}
where $\phi^*_c(t)$ is the optimal value of the convexified version of (\ref{eq:opt-rate}) with the empirical distributions \cite{huangneely_qlamarkovian}, the function $g(\bv{\upsilon}^*(t), \bv{\nu}^*(t)-\bv{\theta}, t)=\sum_m\pi_m(t)g(\bv{\upsilon}, \bv{\nu}-\bv{\theta})$ is introduced for  convenience, 
and $U_{\max}\triangleq U_{tot}(R_{\max}\cdot\bv{1})$.  

Consider Assumption \ref{assumption:bdd-LM} and define: 
\begin{eqnarray*}
\eta_1\triangleq \min\bigg(\sum_m\pi_m\sum_k\varrho^{m}_{k} e_{n,k}^{m},  \sum_m\pi_m(\sum_k\varrho^{m}_{k} e_{n,k}^{m}- h^m_n)\bigg). 
\end{eqnarray*}
Then, we see that for any subset $\script{I}\subseteq \script{N}$, we can construct a policy that ensures (\ref{eq:slackness1}), and ensures (\ref{eq:slackness2}) with:
\begin{eqnarray*}
\hspace{-.3in}&&\sum_m\pi_m\bigg[\sum_k\varrho^{m}_{k} e_{n,k}^{m} - \sum_k\vartheta^{m}_k\sum_{b\in\script{N}^{(o)}_n} P_{[n, b], k}^{m}\bigg]  =  \eta_1, \forall\,n\in\script{I},\\
\hspace{-.3in}&&\sum_m\pi_m\bigg[\sum_k\varrho^{m}_{k} e_{n,k}^{m} - \sum_k\vartheta^{m}_k\sum_{b\in\script{N}^{(o)}_n} P_{[n, b], k}^{m}\bigg]  =  -\eta_1, \forall\,n\notin\script{I}. 
\end{eqnarray*}
Denote $\script{I}_+$ the set of $n$'s with $\nu_n^{*}(t)-\theta_n\geq0$. Using the definition of $g(\bv{\upsilon}, \bv{\nu}, t)$, and using the above policy 
 with $\script{I}=\script{N}/\script{I}_+$, we have that for all time $t\geq T_{\epsilon}$, with probability $1$, 
\begin{eqnarray}
\hspace{-.2in}&&VU_{tot}(\bv{r}) +\eta_0\sum_{n, c} v_n^{(c)*}(t) \nonumber \\
\hspace{-.2in}&& \qquad\qquad +\eta_1\sum_{n\in\script{I}_+}(\nu_n^{*}(t)-\theta_n) -\eta_1 \sum_{n\notin\script{I}_+}(\nu_n^{*}(t)-\theta_n)   \nonumber\\
\hspace{-.2in}&&\leq  g(\bv{\upsilon}^*(t), \bv{\nu}^*(t)-\bv{\theta}, t) \leq  VU_{\max}. \label{eq:dual-bdd}
\end{eqnarray}
(\ref{eq:dual-bdd}) holds since $g(\bv{\upsilon}^*(t), \bv{\nu}^*(t)-\bv{\theta}, t)$ achieves the supremum over all actions.  
From (\ref{eq:dual-bdd}), we have that: 
\begin{eqnarray}
\sum_{n, c} v_n^{(c)*}(t)&\leq& q_d\triangleq \frac{VU_{\max} }{\eta_0}, \label{eq:data-bdd-dual}\\
 \sum_n|\nu_n^{*}(t) -\theta_n| &\leq& q_e\triangleq\frac{  VU_{\max} }{\eta_1}.\label{eq:energy-bdd-dual}
\end{eqnarray}

Denote $\gamma_q\triangleq d_{\max}\mu_{\max}+R_{\max}$ and $\gamma_e\triangleq P_{\max}+h_{\max}$, which are the maximum input rates  into any data queue or energy queue at any time.  
With (\ref{eq:data-bdd-dual}) and (\ref{eq:energy-bdd-dual}),  and the definition of $g(\bv{\upsilon}, \bv{\nu})$, we see that with probability $1$, for all $t\geq T_{\epsilon}$, 
\begin{eqnarray*}
|g_{m}(\bv{\upsilon}^*(t), \bv{\nu}^*(t) -\bv{\theta})| \leq VU_{\max} + q_d\gamma_q+q_e\gamma_e. 
\end{eqnarray*}
Therefore,   
\begin{eqnarray}
\hspace{-.2in}&&|g(\bv{\upsilon}^*(t), \bv{\nu}^*(t)-\bv{\theta} , t) - g(\bv{\upsilon}^*(t), \bv{\nu}^*(t)-\bv{\theta})|\nonumber  \\ 
\hspace{-.2in}&\leq&\sum_{m}| \pi_{m}(t)   - \pi_{m} | |g_{m}(\bv{\upsilon}^*(t), \bv{\nu}^*(t)-\bv{\theta})|\nonumber\\
\hspace{-.2in}&\leq& \sum_{m}| \pi_{m}(t)- \pi_{m}| \times (VU_{\max} + q_d\gamma_q+q_e\gamma_e)   \nonumber \\
\hspace{-.2in}&\leq&\max | \pi_{m}(t)- \pi_{m}|  \times M(VU_{\max} + q_d\gamma_q+q_e\gamma_e). \label{eq:dual-error-bdd}
\end{eqnarray}
Denote $\delta_{m}(t)\triangleq |\pi_{m}(t)  - \pi_{m}  |$ and  $\delta_{tot}\triangleq  \max_{m}\delta_{m}(t) M(VU_{\max} + q_d\gamma_q+q_e\gamma_e)$.  Using (\ref{eq:dual-error-bdd}), we see then: 
\begin{eqnarray}
\hspace{-.4in}&&g(\bv{\upsilon}^*(t), \bv{\nu}^*(t) - \bv{\theta}  )-g(\bv{\upsilon}^*, \bv{\nu}^*) \leq 2\delta_{tot}, \label{eq:distance-rate}
\end{eqnarray}
for otherwise we can use (\ref{eq:dual-error-bdd}) to get: 
\begin{eqnarray*}
&&g(\bv{\upsilon}^*, \bv{\nu}^*, t) - g(\bv{\upsilon}^*(t), \bv{\nu}^*(t)-\bv{\theta} , t)  \\
&\leq& [g(\bv{\upsilon}^*, \bv{\nu}^*) + \delta_{tot}]  - [  g(\bv{\upsilon}^*(t), \bv{\nu}^*(t) -\bv{\theta}   ) - \delta_{tot}  ]\\
&<&0. 
\end{eqnarray*}
This contradicts with the fact that $(\bv{\upsilon}^*(t), \bv{\nu}^*(t))$ is the optimal solution  of  $g(\bv{\upsilon}, \bv{\nu}-\bv{\theta} , t)$. 

Having established (\ref{eq:distance-rate}), we can now use the fact that $g(\bv{\upsilon}, \bv{\nu})$ is a polyhedral function, i.e., $g(\bv{v}^*, \bv{\nu}^*)\leq g(\bv{v}, \bv{\nu})-\rho\|(\bv{v}^*, \bv{\nu}^*)-(\bv{v}, \bv{\nu})\|$,  to obtain:  
\begin{eqnarray*}
\hspace{-.4in}&&\| (\bv{v}^*(t), \bv{\nu}^*(t))  - (\bv{v}^*, \bv{\nu}^*+\bv{\theta}) \| \\
\hspace{-.4in}&&\leq [g(\bv{\upsilon}^*(t), \bv{\nu}^*(t) - \bv{\theta}  )-g(\bv{\upsilon}^*, \bv{\nu}^*)]/\rho \\
\hspace{-.4in}&&\leq 2 \max_{m}\delta_{m}(t) M(VU_{\max} + q_d\gamma_q+q_e\gamma_e)/\rho. \nonumber
\end{eqnarray*}
Finally, we can use a similar argument as in Appendix H of \cite{huang-learning-sig-14} to show that, when $V$ is large enough such that $\frac{2}{3}\log(V)V^{-c/2}\leq1$, at time $T_L=V^c$, 
\begin{eqnarray}
\prob{\max\delta_m(t) \geq \frac{4\log(V)}{V^{c/2}} } \leq Me^{-4[\log(V)]^2}. \label{eq:prob-error-bdd}
\end{eqnarray}
Therefore, with probability at least $1-Me^{-4[\log(V)]^2}$, we have $\max\delta_{m}(t) \leq \frac{4\log(V)}{V^{c/2}}$. Thus, 
\begin{eqnarray}
\hspace{-.4in}&&\| (\bv{v}^*(t), \bv{\nu}^*(t))  - (\bv{v}^*, \bv{\nu}^*+\bv{\theta}) \| \\
\hspace{-.4in}&&\leq \frac{8\log(V)M(VU_{\max} + q_d\gamma_q+q_e\gamma_e)}{V^{c/2}  \rho} \nonumber\\
\hspace{-.4in}&&=\Theta(V^{1-\frac{c}{2}}\log(V)). 
\end{eqnarray}
This proves the lemma. 
\end{proof}

\section*{Appendix B -- Proof of Theorem \ref{theorem:lem}}
Here we prove Theorem \ref{theorem:lem}. We will use the following lemma and theorems in our analysis.  
\begin{lemma}\label{lemma:q-avgrate}
 \cite{huang-learning-sig-14} Let $Q(t)$ be the size of a single queue with dynamics $Q(t+1)=[Q(t)-\mu(t)]^++A(t)$. Suppose $0\leq \mu(t), A(t)\leq\delta_{\max}=\Theta(1)$ for all $t$ and that the queue is stable. Then,  
\begin{eqnarray}
\overline{\mu(t)} - \overline{A(t)} \leq \delta_{\max}\prob{Q(t)<\delta_{\max}}. \label{eq:excessrate}
\end{eqnarray}
Here $\overline{x(t)}\triangleq \lim_{T\rightarrow\infty}\frac{1}{T}\sum_{t=0}^{T-1}\expect{x(t)}$. $\Diamond$ 
\end{lemma}

\begin{theorem}\label{theorem:dual-fav} \cite{huangneely_qlamarkovian} 
Let $(\bv{v}^*, \bv{\nu}^*)$ be an optimal solution of (\ref{eq:dual}). Then, $g(\bv{v}^*, \bv{\nu}^*)=\phi^*\geq VU_{tot}(\bv{r}^*)$. $\Diamond$
\end{theorem}

%

\begin{theorem} \label{theorem:attraction} Suppose that the dual function $g(\bv{v}^*, \bv{\nu}^*)$  is  polyhedral with $\rho = \Theta(1)>0$, i.e., independent of $V$, and has a unique optimal $(\bv{v}^*, \bv{\nu}^*)$. Then, under the LEM algorithm without learning and augmenting, there exist $\Theta(1)$ constants $\xi$, $K$, and $D$, such that,  
\begin{eqnarray}
\script{P}(D, b)\leq \xi e^{-Kb},\label{eq:pm_ineq}
\end{eqnarray}
where $\script{P}(D, b)$ is defined as:
\begin{eqnarray}
\hspace{-.3in}&&\script{P}(D, b)\triangleq\label{eq:pm_def}\\
\hspace{-.3in}&&\,\,\,\limsup_{t\rightarrow\infty}\frac{1}{t}\sum_{\tau=0}^{t-1}\prob{\| (\bv{Q}(t), \bv{E}(t)-\bv{\theta})  - (\bv{v}^*, \bv{\nu}^*)  \|>D+b}. \nonumber
\end{eqnarray}
\end{theorem}
\begin{proof} (Theorem \ref{theorem:attraction}) 
Similar to the proof of Theorem 1 in \cite{huangneely_dr_tac}. Omitted for brevity. 
\end{proof}

If a steady state distribution exists for $\bv{Q}(t)$ and $\bv{E}(t)$, then $\script{P}(D, b) = \lim_{t\rightarrow\infty}\prob{ \| (\bv{Q}(t), \bv{E}(t)-\bv{\theta})  - (\bv{v}^*, \bv{\nu}^*)  \|>D+b }$. 
We now present the proof for Theorem \ref{theorem:lem}. We carry out the proof in the order of Part (c) - Part (a) - Part (b). 
Our analysis is very different from those in \cite{huang-learning-sig-14} and  \cite{huangneely-energy-ton13}, due to the energy underflow constraint (\ref{eq:energycond}) and learning. Specifically, in \cite{huangneely-energy-ton13}, the energy availability is shown to be never active; in our case, we introduce packet dropping to ensure zero energy outage.  


\begin{proof}   (Theorem \ref{theorem:lem} - Part (c) \textbf{Utility}) 
To start,   first note that the dropping mechanism  changes the effective energy queue evolution to: 
\begin{eqnarray}
\hspace{-.3in}&&E_n(t+1) =  \big(E_n(t) - \sum_{b\in\script{N}^{(o)}_n}P_{[n, b]}(t)\big)^+ +e_n(t). \label{eq:Edynamic2} 
\end{eqnarray}
It is important to note that (\ref{eq:drift3}) still holds under (\ref{eq:Edynamic2}). Thus, by comparing the RHS of (\ref{eq:drift3}) and LEM, 
we get that:  
\begin{eqnarray}
\hspace{-.3in} &&\Delta(t) + \Delta_A(t)-V\expect{\sum_{n, c} U^{(c)}_n(R^{(c)}_n(t))\left.|\right. \bv{Y}(t)} \label{eq:drift4}\\
\hspace{-.3in} &&\leq B - V\expect{\sum_{n, c} U^{(c)}_n(R^{(c), alt}_n(t))\left.|\right. \bv{Y}(t)} \nonumber\\
\hspace{-.3in} &&\quad+\sum_{n\in\script{N}}(\hat{E}_n(t)-\theta_n)\expect{e^{alt}_n(t) -\sum_{b\in\script{N}^{(o)}_n} P_{[n, b]}^{alt}(t)\left.|\right. \bv{Y}(t)}  \nonumber\\
\hspace{-.3in} && \quad   - \expect{ \sum_{n, c} \hat{Q}^{(c)}_n(t)\big[ \sum_{b\in\script{N}^{(o)}_n}\mu_{[n,b]}^{(c), alt} -  \sum_{b\in\script{N}^{(in)}_n}\mu_{[a, n]}^{(c), alt} \nonumber\\
\hspace{-.3in} &&\qquad\qquad\qquad\qquad\qquad\qquad\qquad\quad  - R^{(c), alt}_n(t)\big] \left.|\right. \bv{Y}(t)}. \nonumber
\end{eqnarray}
Here we have rearranged the terms and  ``$alt$'' stands for ``alternative,'' i.e., the drift expression under LEM holds when we plug any alternative control policy into the RHS. 

By comparing $\mathtt{LEM}$ and the dual function $g(\bv{v}, \bv{\nu})$, we see that the RHS of 
(\ref{eq:drift4}) under $\mathtt{LEM}$ indeed equals to $B-g(\hat{\bv{Q}}(t), \hat{\bv{E}}(t)-\bv{\theta})$. Thus, applying Theorem \ref{theorem:dual-fav}, we obtain: 
 \begin{eqnarray}
\hspace{-.3in} && \Delta(t) +\Delta_A(t) -V\expect{ \sum_{n, c} U^{(c)}_n(R^{(c)}_n(t)) \left.|\right.\bv{Y}(t)}  \nonumber\\
\hspace{-.3in} &\leq& B-g(\hat{\bv{Q}}(t), \hat{\bv{E}}(t)-\bv{\theta})\leq B - VU_{tot}(\bv{r}^*). \nonumber
\end{eqnarray}
Taking expectations over $\bv{Y}(t)$ and summing the above over $t=0, ..., T-1$, we have:
\begin{eqnarray*}
\hspace{-.3in} &&\expect{L(T)-L(0)} + \sum_{t=0}^{T-1} \expect{ \Delta_A(t)} \\
\hspace{-.3in} &&- V\sum_{t=0}^{T-1} \expect{ \sum_{n, c} U^{(c)}_n(R^{(c)}_n(t))} \leq T B - T VU_{tot}(\bv{r}^*). 
\end{eqnarray*}
Rearranging  terms,  using the fact that $L(t)\geq0$ and $L(0)=0$, and dividing both sides by $VT$, 
we get:
\begin{eqnarray*}
\hspace{-.3in} && \frac{1}{T}\sum_{t=0}^{T-1} \expect{ \sum_{n, c} U^{(c)}_n(R^{(c)}_n(t))} \\
\hspace{-.3in} && \qquad\qquad\geq U_{tot}(\bv{r}^*) -  \frac{B}{V} - \frac{1}{VT}\sum_{t=0}^{T-1} \expect{ \Delta_A(t)}. 
\end{eqnarray*}
Using Jensen's inequality and taking a limit as $T\rightarrow\infty$,  
\begin{eqnarray*}
\hspace{-.3in} && \sum_{n, c} U^{(c)}_n(\overline{r}^{nc})\geq U_{tot}(\bv{r}^*) - \frac{B}{V} - \lim_{T\rightarrow\infty}\frac{1}{VT}\sum_{t=0}^{T-1} \expect{ \Delta_A(t)}. 
\end{eqnarray*}

To prove (\ref{eq:utility-perf}), it remains to show that  (i) the last term, $\lim_{T\rightarrow\infty}\frac{1}{VT}\sum_{t=0}^{T-1} \expect{ \Delta_A(t)}=O(1/V)$, and (ii) the fraction of time dropping  happens is $O(1/V)$. 

We start with (i).  Recall that $\Delta_A(t)$ is defined as follows:  
\begin{eqnarray}
\hspace{-.4in}&&\Delta_A(t)\triangleq - \expect{\sum_n\xi_{Q,n}^{(c)} [  \sum_{b\in\script{N}^{(o)}_n}\mu_{[n, b]}^{(c)}(t) \label{eq:augmenting-recall}\\
\hspace{-.4in}&&\qquad\qquad\qquad\quad\quad\quad - \sum_{b\in\script{N}^{(in)}_n}\mu_{[a, n]}^{(c)}(t)  - R_n^{(c)}(t)  ]   \left.|\right. \bv{Y}(t)}\nonumber\\
\hspace{-.4in}&&\qquad \quad\quad - \expect{\sum_n\xi_{E, n} [ \sum_{b\in\script{N}^{(o)}_n} P_{[n, b]}(t) - e_n(t) ]   \left.|\right. \bv{Y}(t)}. \nonumber
\end{eqnarray}
Notice that we only need to consider $T_L\leq t\leq T$, since $\bv{\xi}_Q=\bv{\xi}_E=\bv{0}$ before time $T_L$. Using Lemma \ref{lemma:beta-conv}, we see that with probability $1- O(\frac{1}{V^{4\log(V)}})$, (\ref{eq:approx-guess}) holds at time $t=T_L$. This implies that when $V$ is large, with probability $1- O(\frac{1}{V^{4\log(V)}})$, 
\begin{eqnarray}
\hspace{-.4in}&&v_{n}^{(c)*}  - \frac{3}{2}V^{1-\frac{c}{2}}\log(V)^2 \leq  \xi_{Q,n}^{(c)} \label{eq:guess-bdd1}\\
\hspace{-.4in}&&\qquad \qquad\qquad\qquad\quad\leq  v_{n}^{(c)*}  - \frac{1}{2}V^{1-\frac{c}{2}}\log(V)^2,\,\,\forall\,\,n, c \nonumber \\
\hspace{-.4in}&&\nu_n^* +\theta_n   - \frac{3}{2}V^{1-\frac{c}{2}}\log(V)^2 \leq  \xi_{E, n}\label{eq:guess-bdd2} \\
\hspace{-.4in}&&\qquad \qquad\qquad\qquad\quad  \leq \nu_n^* +\theta_n   - \frac{1}{2}V^{1-\frac{c}{2}}\log(V)^2 ,\,\,\forall\,\,n. \nonumber
\end{eqnarray}
Combined with (\ref{eq:pm_ineq}) and the definitions of $\hat{\bv{Q}}(t)$ and $\hat{\bv{E}}(t)$, they   imply that, with probability $1- O(\frac{1}{V^{4\log(V)}})$,  in steady state: 
\begin{eqnarray*}
\prob{ Q_n^{(c)}(t) \leq \frac{1}{2}V^{1-\frac{c}{2}}\log(V)^2} \leq \xi e^{-K(\frac{1}{2}V^{1-\frac{c}{2}}\log(V)^2 - D)} \\
\prob{ E_n(t) \leq \frac{1}{2}V^{1-\frac{c}{2}}\log(V)^2} \leq \xi e^{-K(\frac{1}{2}V^{1-\frac{c}{2}}\log(V)^2 - D)}.  
\end{eqnarray*}
Thus, for a sufficiently large $V$ such that $\frac{1}{2}V^{1-\frac{c}{2}}\log(V)^2\geq D+P_{\max}+\mu_{\max} + \log(V)^2/K$, one has: 
\begin{eqnarray}
&&\prob{ Q_n^{(c)}(t) \leq \mu_{\max}} \leq \xi V^{-\log(V)}, \label{eq:rare-q} \\
&&\prob{ E_n(t) \leq P_{\max}} \leq \xi V^{-\log(V)}.  \label{eq:rare-e}
\end{eqnarray}
Using Lemma \ref{lemma:q-avgrate}, we conclude then: 
\begin{eqnarray}
\hspace{-.2in}  \sum_{b\in\script{N}^{(o)}_n}\overline{\mu_{[n, b]}^{(c)}(t)}  - \sum_{b\in\script{N}^{(in)}_n}\overline{\mu_{[a, n]}^{(c)}(t)}  - \overline{R_n^{(c)}(t)} &=&O(\frac{1}{V^2}) \nonumber\\
\hspace{-.2in} \sum_{b\in\script{N}^{(o)}_n} \overline{P_{[n, b]}(t)} - \overline{e_n(t)} &=&  O(\frac{1}{V^2}). \nonumber
\end{eqnarray}
Together with (\ref{eq:guess-bdd1}) and (\ref{eq:guess-bdd2}), and that $(\bv{v}^*, \bv{\nu}^*+\bv{\theta})=\Theta(V\log(V))$, the above shows that with probability $1- O(\frac{1}{V^{4\log(V)}})$, the last term $\lim_{T\rightarrow\infty}\frac{1}{VT}\sum_{t=0}^{T-1} \expect{ \Delta_A(t)}=O(1/V)$. 

It remains  to show that dropping happens rarely. We first use (\ref{eq:rare-e}) to conclude that the fraction of time the energy queue does not have enough power after $t=T_L$ is $O(1/V^2)$.  
To also see that no energy outage occurs before $T_L$, note that before $T_L$, we have $Q_n^{(c)}(t)=O(V^{c})$. This is so since a queue can increase by at most $\gamma_q=R_{\max}+d_{\max}\mu_{\max} = \Theta(1)$ in every time slot. Thus, for any link, we have $W_{[n, b]}^{(c)}(t) \leq \beta V$. Since $\theta_n=V\log(V)$ for all $n$, we see that if $E_n(t)\leq P_{\max}$, then $E_n(t)-\theta_n < -\kappa\beta V$. Now, suppose at time $t$ the optimal power allocation $\bv{P}$  has a nonzero component $P_{[n, b]}$ for some node $n$ with $E_n(t)<P_{\max}$. Then, we can construct  vector $\hat{\bv{P}}$  by setting only $P_{[n, b]}$ to zero. Doing so, we get: 
\begin{eqnarray*}
\hspace{-.3in} &&\quad G(\bv{P}) - G(\hat{\bv{P}}) \\
\hspace{-.3in} && =\sum_{n}\sum_{b\in\script{N}^{(o)}_n}\big[\mu_{[n, b]}(\bv{S}(t), \bv{P}) - \mu_{[n, b]}(\bv{S}(t), \hat{\bv{P}})\big]W_{[n, b]}(t) \\
\hspace{-.3in} &&\qquad\qquad\qquad\qquad\qquad\qquad\qquad+(E_n(t)-\theta_n)P_{[n, b]}\nonumber\\
\hspace{-.3in} && \leq \big(\mu_{[n, b]}(\bv{S}(t), \bv{P}) - \mu_{[n, b]}(\bv{S}(t), \hat{\bv{P}})\big)W_{[n, b]}(t)   \\
\hspace{-.3in} &&\qquad\qquad\qquad\qquad\qquad\qquad\qquad +(E_n(t)-\theta_n)P_{[n, b]}\nonumber\\
\hspace{-.3in} &&<  \beta V\kappa P_{[n, b]} -\kappa\beta VP_{[n, b]} = 0. 
\end{eqnarray*}
Here the last two inequalities follow from the properties of $\bv{\mu}(t)$. This contradicts with the fact that $\bv{P}$ is the optimal vector at time $t$.  
Thus, no energy outage will occur before time $T_L$. This establishes (ii). 

Now note that  packet dropping happens only when $E_n(t)<P_{\max}$. Thus, using the results from above, we conclude that the average packet dropping rate is also $O(1/V^2)$. Since the utility functions have maximum derivative $\beta=\Theta(1)$, this implies that the utility loss due to packet dropping is $O(1/V^2)$, and  completes the proof for   utility.

(Part (a) - \textbf{Queue}) We now prove the queueing results. Using (\ref{eq:pm_ineq}), (\ref{eq:guess-bdd1}), and (\ref{eq:guess-bdd2}), we see that in steady state,  
\begin{eqnarray}
\hspace{-.4in}&&\prob{ Q_n^{(c)}(t) \geq \frac{3}{2}V^{1-\frac{c}{2}}\log(V)^2+D +b}  \leq \xi e^{-Kb}  \\
\hspace{-.4in}&&\prob{ E_n(t) \geq \frac{3}{2}V^{1-\frac{c}{2}}\log(V)^2+D+b} \leq \xi e^{-Kb}.  
\end{eqnarray}
Note that these are exactly (\ref{eq:queue-prob-bdd}) and (\ref{eq:energy-prob-bdd}). From these results, we see that  (\ref{eq:data-queue-bound}) and (\ref{eq:energy-queue-bound}) follow.

(Part (b) - \textbf{Delay}) To prove Part (b), we note that LEM is indeed a Lyapunov drift-based algorithm with the LIFO queue discipline (learning happens  only at $t=T_L$). Thus, its delay performance follows from Theorem 4 in \cite{huang-lifo-ton}. 
 \end{proof}

\section*{Appendix C -- Proof of Theorem \ref{theorem:lem-conv}}
Here we prove Theorem \ref{theorem:lem-conv}. 
\begin{proof} (Theorem \ref{theorem:lem-conv})
First, using Lemma \ref{lemma:beta-conv}, we see that at time $t=T_L=V^c$, with probability $1-O(\frac{1}{V^{\log(V)}})$, we have:
\begin{eqnarray}
\| (\bv{v}(t), \bv{\nu}(t))  - (\bv{v}^*, \bv{\nu}^*+\bv{\theta})  \| \leq \Theta(V^{1-\frac{c}{2}}\log(V)). 
\end{eqnarray}
Using the definitions of $\hat{\bv{Q}}(t)$ and $\hat{\bv{E}}(t)$, this implies that: 
\begin{eqnarray}
\| (\hat{\bv{Q}}(t), \hat{\bv{E}}(t))   - (\bv{v}^*, \bv{\nu}^*+\bv{\theta})  \| \leq \Theta(V^{1-\frac{c}{2}}\log(V)^2). 
\end{eqnarray}
Then, note that for all $t>T_L$, LEM is the same as a pure drift-based control policy. Hence, using Lemma 2 in \cite{huangneely_dr_tac},  we have that whenever $ \| (\hat{\bv{Q}}(t), \hat{\bv{E}}(t))   - (\bv{v}^*, \bv{\nu}^*+\bv{\theta})  \| \geq D$, where $D$ is defined in Theorem \ref{theorem:attraction},  
\begin{eqnarray}
\hspace{-.4in}&&\expect{\| (\hat{\bv{Q}}(t+1), \hat{\bv{E}}(t+1))   - (\bv{v}^*, \bv{\nu}^*+\bv{\theta})  \| \left.|\right. \hat{\bv{Q}}(t), \hat{\bv{E}}(t)  } \\
\hspace{-.4in}&&\qquad\qquad\qquad\qquad  \leq \| (\hat{\bv{Q}}(t), \hat{\bv{E}}(t))   - (\bv{v}^*, \bv{\nu}^*+\bv{\theta})  \| - \epsilon_0, \nonumber
\end{eqnarray}
for some $\epsilon_0=\Theta(1)$. Thus, applying Theorem 4 in \cite{huang-learning-sig-14}, we get that the expected time it takes for $(\hat{\bv{Q}}(t), \hat{\bv{E}}(t))$ to get into within $D$ distance of $(\bv{v}^*, \bv{\nu}^*+\bv{\theta})$ is no more than $\Theta((V^{1-\frac{c}{2}}\log(V)^2-D)/\epsilon_0)$. 

Combining this with the fact that the learning time is $T_L=V^c$, we have: 
\begin{eqnarray}
\expect{T_{D}} &=& \Theta((V^{1-\frac{c}{2}}\log(V)^2-D)/\epsilon_0 + V^c).
\end{eqnarray}
This completes the proof of Theorem \ref{theorem:lem-conv}. 
\end{proof}

\end{document}